\documentclass[12pt, oneside]{amsart}

\usepackage{amssymb, amsmath, amsthm,geometry}
\usepackage{mathrsfs}  

\usepackage[lithuanian,english]{babel}

\usepackage{geometry}

\usepackage{bm}
\usepackage{subfig,tikz}
\usepackage{wrapfig}
\usepackage{color}
\usepackage{xspace}
\usetikzlibrary{shapes.geometric}

\usepackage{graphicx}
\usepackage[colorlinks=true, citecolor=blue, linkcolor=red, urlcolor=red]{hyperref}
 \usepackage{xcolor}
\usepackage{hyperref}

\usepackage{lineno}

\newtheorem{thm}{Theorem}[section]
\newtheorem{lem}[thm]{Lemma}
\newtheorem{cor}[thm]{Corollary}
\newtheorem{prop}[thm]{Proposition}
\newtheorem{defn}[thm]{Definition}
\newtheorem{hypothesis}[thm]{Hypothesis}
\theoremstyle{remark}
\newtheorem{rmk}{Remark}

\numberwithin{equation}{section}
\newcommand{\bel}{\begin{equation} \label}
\newcommand{\ee}{\end{equation}}
\def\beq{\begin{equation}}
\def\eeq{\end{equation}}
\newcommand{\bea}{\begin{eqnarray}}
\newcommand{\eea}{\end{eqnarray}}
\newcommand{\beas}{\begin{eqnarray*}}
\newcommand{\eeas}{\end{eqnarray*}}
\newcommand{\pd}{\partial}

\newcommand{\floor}[1]{\lfloor #1 \rfloor}

\newcommand{\R}{\mathbb{R}}

\newcommand{\N}{\mathbb{N}}

\newcommand{\cO}{\mathcal{O}}

\newcommand{\M}{\mathcal{M}}

\newcommand{\dis}{\mathrm{dist}\,}

\newcommand{\supp}{\mathrm{supp}\,}  

\allowdisplaybreaks

\def\epsilon{\varepsilon}
\def\phi {\varphi}

\def\p{\partial}

\newcommand{\tnorm}{\vert\hspace{-0.3mm}\Vert }

\renewcommand{\leq}{\leqslant}
\renewcommand{\geq}{\geqslant}

\numberwithin{equation}{section}

\title[A Fully Discrete numerical control method for the wave equation]{A Fully Discrete numerical Control Method For The Wave equation}
\author[Erik Burman]{Erik Burman}
\address{Department of Mathematics, University College London, London, UK-WC1E  6BT, United Kingdom}
\email{e.burman@ucl.ac.uk}
\author[Ali Feizmohammadi]{Ali Feizmohammadi}
\address{Department of Mathematics, University College London, London, UK-WC1E  6BT, United Kingdom}
\email{a.feizmohammadi@ucl.ac.uk}
\author[Lauri Oksanen]{Lauri Oksanen}
\address{Department of Mathematics, University College London, London, UK-WC1E  6BT, United Kingdom}
\email{l.oksanen@ucl.ac.uk}

\date{}
\begin{document}
\maketitle
\begin{abstract}
We present a fully discrete finite element method for the interior null controllability problem subject to the wave equation. For the numerical scheme, piece-wise affine continuous elements in space and finite differences in time are considered. We show that if the sharp geometric control condition holds, our numerical scheme yields the optimal rate of convergence with respect to the space-time mesh parameter $h$. The approach is based on the design of stabilization terms for the discrete scheme with the goal of minimizing the computational error. 
\end{abstract}
\section{Introduction}
\label{intro}
We consider the now classical interior null controllability problem for the wave equation formulated as follows.
Let $T>0$, $\Omega \subset \R^n$ with $n\in \{2,3\}$ be a connected bounded open set with smooth boundary and finally let $\omega \subset \Omega$ be an open set. We define $\mathcal{M}=(0,T)\times \Omega$ and $\mathcal{O}=(0,T)\times \omega$ and for each $$(g_0,g_1,U) \in H^1_0(\Omega)\times L^2(\Omega) \times L^2(\M),$$ consider the unique weak solution 
$$u \in \mathcal C(0,T;H^1_0(\Omega)) \cap \mathcal C^1(0,T;L^2(\Omega))$$ 
of the following initial boundary value problem (IBVP):  
\bel{pf}
\begin{aligned}
\begin{cases}
\Box u=\partial^2_t u - \Delta u = \chi_{\omega}U, 
&\forall (t,x) \in \M,
\\
u(t,x) =0,
&\forall (t,x)\in (0,T) \times \p \Omega,
\\
u(0,x)= g_0,\, \p_t u(0,x)=g_1,
&\forall x \in \Omega.
\end{cases}
    \end{aligned}
\ee
Here $\chi_{\omega}$ is a suitable non-negative smooth function that is localized in $\omega$ and is independent of the time parameter $t$. 


{\em The null controllability problem} consists of determining a control function $U_*$, such that the solution $u$ to equation \eqref{pf} with $U=U_*$ satisfies
\bel{ctrl}
(u(T,x),\p_t u(T,x))=(0,0) \quad \forall x \in \Omega.
\ee

This paper is concerned with a numerical scheme for solving the null
controllability problem \eqref{pf}$-$\eqref{ctrl}, based on the Finite
Element Method (FEM). In particular, we will prove optimal rate of
convergence of the error in the $H^1$-norm of the state variable $u$,
with respect to the space-time mesh parameter, assuming only the
geometric control condition by Bardos, Lebeau and Rauch \cite{BLR,BLRII}. To
our knowledge, the present result is first one giving optimal
convergence rate in general geometries in dimensions two and three. 

\subsection{The geometric control condition and observability estimates}
We begin with recalling the geometric control condition by Bardos, Lebeau and Rauch:
\begin{defn}[See \cite{BLRII,LLTT}]
We say that $\tilde{\mathcal O}=(0,T)\times \tilde\omega$ satisfies the geometric control condition, if every compressed generalized bicharacteristic intersects the set $\tilde{\mathcal O}$.
\end{defn}
This roughly states that all geometric optic rays propgating in $\M$ must intersect the region $\tilde{\mathcal O}$, taking into account possible reflections of the rays at the boundary. Next, we recall the following observability estimate originating from \cite{BLR,BLRII}. The formulation here is based on \cite[Proposition 1.2]{LLTT} and is stated as it appears in \cite[Theorem 2.2]{BFO}:
\begin{thm}(Interior observability estimate)
\label{observability}
Let $T>0$, $\tilde{\omega} \subset \Omega$. Suppose that the set $\tilde{\mathcal{O}}=[0,T]\times \tilde{\omega}$ satisfies the geometric control condition. Let $U \in L^2(\M)$ with $U|_{(0,T)\times\p \Omega}\in L^2((0,T)\times \p \Omega)$ and $\Box U\in H^{-1}(\M)$, where $H^{-1}(\M)$ denotes the topological dual of $H^1_0(\M)$. Then,
$$ U \in \mathcal C^1(0,T;H^{-1}(\Omega))\cap \mathcal C(0,T;L^2(\Omega).$$
Moreover, there exists $C_0>0$ such that the following estimate holds:
$$\sup_{t \in [0,T]}(\|U(t,\cdot)\|_{L^2(\Omega)}+\|\p_t U(t,\cdot)\|_{H^{-1}(\Omega)}) \leq C_0(\|U\|_{L^2(\tilde{\mathcal O})}+\|\Box U\|_{H^{-1}(\M)}+\|U\|_{L^2((0,T)\times\p \Omega)}).$$
\end{thm} 
Observability estimates are one of the key tools in the study of the null controllability problem for the wave equation \cite{MZ}. Although alternative geometric conditions are also available for obtaining such an estimate (see for example \cite{DZZ,Miller}), it is important to note that the geometric control condition is sharp in the sense that it is both necessary and sufficient for obtaining an observability estimate.



\subsection{Continuum null controllability problem} 
\label{continuum}
We recall the classical approach in showing the existence of a control function $U$ that originates from \cite{Lions1988}. Although in general the problem of determining a control function $U$ solving \eqref{pf}$-$\eqref{ctrl} is non-unique, we may look for controls with additional constraints. The standard approach is to choose a control that additionally satisfies the (backward) wave equation as well, that is:
\bel{dual}
\begin{aligned}
\begin{cases}
\Box U=0, 
&\forall (t,x) \in \M,
\\
U(t,x) =0,
&\forall (t,x)\in (0,T) \times \p \Omega,
\\
U(T,x)= U_0, \, \p_t U(T,x)=U_1,
&\forall x \in \Omega
\end{cases}
    \end{aligned}
\ee
for some $(U_0,U_1) \in L^2(\Omega)\times H^{-1}(\Omega)$. 

We recall from \cite[Theorem 2.3]{LLT} that \eqref{dual} has a unique solution $U$ in the energy space
$$ \mathcal C^1(0,T;H^{-1}(\Omega))\cap \mathcal C(0,T;L^2(\Omega)).$$
Observe that given any solution $u$ to equation \eqref{pf}, and any solution $V$ to equation \eqref{dual}, we have:
\bel{}
\begin{aligned}
\int_0^T (\chi_{\omega}(\cdot)&U(\tau,\cdot),V(\tau,\cdot))_{L^2(\Omega)} \,d\tau=(\p_t u(T,\cdot),V(T,\cdot))_{L^2(\Omega)}-(\p_t u(0,\cdot),V(0,\cdot))_{L^2(\Omega)}\\
&-\langle u(T,\cdot),\p_t V(T,\cdot) \rangle_{H^1_0(\Omega)\times H^{-1}(\Omega)}+\langle u(0,\cdot),\p_t V(0,\cdot) \rangle_{H^1_0(\Omega)\times H^{-1}(\Omega)}.
\end{aligned}
\ee
We deduce that equations \eqref{pf}$-$\eqref{ctrl} hold if and only if the following identity holds for any solution $V$ to the wave equation \eqref{dual}: 
\bel{critical}
\int_0^T (\chi_{\omega}(\cdot)U(\tau,\cdot),V(\tau,\cdot))_{L^2(\Omega)} \,d\tau=-(g_1,V(0,\cdot))_{L^2(\Omega)}+\langle g_0,\p_t V(0,\cdot) \rangle_{H^1_0(\Omega)\times H^{-1}(\Omega)}.
\ee
Under the additional assumption that the control function $U$ satisfies the wave equation \eqref{dual}, equation \eqref{critical} is equivalent to the Euler-Lagrange equation for the Lagrangian functional 
\bel{cont-Lagr}
\mathcal{J}(U_0,U_1)= \frac{1}{2}\int_0^T\int_\Omega \chi_\omega |U|^2 \,dx\,dt -(g_1,U(0,\cdot))_{L^2(\Omega)}+\langle g_0,\p_t U(0,\cdot) \rangle_{H^1_0(\Omega)\times H^{-1}(\Omega)},
\ee
where, for each $(U_0,U_1) \in L^2(\Omega)\times H^{-1}(\Omega)$, $U$ denotes the unique solution to equation \eqref{dual} with this final datum. 

To summarize thus far, let $(U_{*,0}, U_{*,1})$ be a minimizer (if it exists) for the functional $\mathcal J$. Then, the solution $U_*$ to \eqref{dual} with this final datum yields a control function that drives the solution $u_*$ of \eqref{pf} with source term $\chi_\omega U_*$ from $(g_0,g_1)$ to $(0,0)$. In fact one can show that $U_*$ is the control function solving \eqref{pf}--\eqref{ctrl} with minimal $\|\sqrt{\chi_\omega}U\|_{L^2(\M)}$ norm.

We will now briefly recall how the observability estimate in Theorem~\ref{observability} proves existence of a unique minimizer for $\mathcal J$. Let us first consider the classical context where $\mathcal O$ satisfies the geometric control condition and additionally that $\chi_\omega$ is simply the characteristic function of the set $\omega$. In this case, Theorem~\ref{observability} implies that the functional $\mathcal J$ is coercive and strictly convex (see for example \cite[Theorem 2.4]{MZ}). Together with the continuity of $\mathcal J$ it follows that, in this setting, there exists a unique minimizer $(U_{*,0},U_{*,1})$ in the space $L^2(\Omega)\times H^{-1}(\Omega)$. 

It is in fact quite common in the literature to let $\chi_\omega$ be the characteristic function of $\omega$ as above. In this case, the control function $U_*$ suffers from low regularity that makes the task of numerical approximation and derivation of convergence rates challenging. Already in the seminal work \cite{BLRII}, a theory for smoother boundary controls for the wave equation were studied. In \cite{EZ,EZ2}, the authors studied interior controls for the wave equation and in particular it was proved that one can construct smoother control functions by simply imposing some smoothness conditions on the initial datum $(g_0,g_1)$ and using a sufficiently smooth cut-off function $\chi_\omega$. Let us recall their approach for the continuum problem. We need the following definition.
\begin{defn}
\label{compat}
For each $s\in\N$, we say that $(y_0,y_1) \in \mathcal{D}((-\Delta)^s)$, if the following conditions are satisfied:
\begin{itemize}
\item[(i)]{$(y_0,y_1) \in H^{s+1}(\Omega)\times H^{s}(\Omega)$,}
\item[(ii)]{$((-\Delta)^j y_0)|_{\p \Omega}=0, \quad \text{for $j=0,1,\ldots,\floor{\frac{s}{2}+\frac{1}{4}}$}$,}
\item[(iii)]{$((-\Delta)^j y_1)|_{\p \Omega}=0,\quad \text{for $j=0,1,\ldots,\floor{\frac{s}{2}-\frac{1}{4}}$}$.}
\end{itemize}
\end{defn}


We now recall \cite[Theorem 4]{EZ} and \cite[Theorem 1.6]{EZ2} to state some regularity results for the controls that are obtained when smoother cut-off functions are used.
\begin{thm}
\label{smoothness}
Let $s \in \N$. Suppose that $\chi_\omega$ is a non-negative smooth function localized in $\omega$ that maps $\mathcal{D}((-\Delta)^s)$ to itself. Assume also that the functional $\mathcal J$ given by \eqref{cont-Lagr} is coercive and strictly convex. Given any initial datum $(g_0,g_1) \in \mathcal{D}((-\Delta)^s)$, let $(U_{*,0},U_{*,1})$ denote the unique minimizer for $\mathcal J$ on $L^2(\Omega)\times H^{-1}(\Omega)$. Then:$$(U_{*,0},U_{*,1}) \in \mathcal D((-\Delta)^{s-1}).$$ 
Moreover, the following estimate holds:
\bel{smoothness_eq} \|U_*\|_{\mathscr{X}_s(\M)}+\|u_*\|_{\mathscr{X}_{s+1}(\M)} \leq C \|(g_0,g_1)\|_{H^{s+1}(\Omega)\times H^s(\Omega)},\ee
where $U_*$ is the unique solution to \eqref{dual} with final datum $(U_{*,0},U_{*,1}) $ and $u_*$ is the unique solution to \eqref{pf} with source $\chi_\omega U_*$. Here, $C>0$ is a constant depending only on $\M$, $\omega$, $\chi_\omega$, $s$ and $\mathscr{X}_s(\M)$ denotes the Banach space 
$\mathscr{X}_s(\M) = \bigcap_{k=0}^{s} \mathcal C^k(0,T;H^{s-k}(\Omega)).$
\end{thm} 

Note that this theorem gives a continuum solution $(u_*,U_*)$ to the null controllability problem \eqref{pf}--\eqref{dual}, with smoothness properties given by \eqref{smoothness_eq}. In this paper, we will need to apply Theorem~\ref{smoothness} with $s=3$. We will therefore begin with defining an admissibility condition for the set $\mathcal O$, based on the geometric control condition, followed by the admissible choices of the cut-off function $\chi_\omega$ so that the assumptions of Theorem~\ref{smoothness} are satisfied. 

\begin{hypothesis}[Admissibility condition for $\mathcal O$]
\label{hypo} There exists $\delta>0$ sufficiently small, such that the set $(0,T)\times \omega_\delta$ satisfies the geometric control condition, where:
$$\omega_\delta=\{x \in \omega\,|\,\dis{(x,\p\omega\setminus\p\Omega)}>\delta\}.$$
\end{hypothesis}

Next, assuming that the set $\mathcal O$ satisfies the admissibility condition above, we require that our cut-off function $\chi_\omega$ satisfies the following
three properties
\begin{equation}\label{cutoff_conditions}
  \begin{cases}
    \text{(i)\,$\chi_\omega \in \mathcal C^{\infty}(\bar{\Omega};[0,\infty))$ and $\chi_\omega=0$ on the set $\Omega \setminus \omega$,}\\
    \text{(ii)\,$\chi_\omega=1$ on the set $\omega_\delta$,}\\ 
    \text{(iii)\,$(\p_\nu^k \chi_\omega)|_{\p\Omega}=0$ for $k=1,2,$ where $\nu$ denotes the unit normal vector to $\p \Omega$.}
  \end{cases}
\end{equation}

We will show in Appendix~\ref{const_cut} that one can always construct such cut-off functions. As an example, we note that in the special case that $\p \omega \cap \p \Omega =\emptyset$, the cut-off function can be chosen as any function $\chi_\omega\in\mathcal C^{\infty}_c(\omega;[0,1])$ that satisfies:
\[
\chi_\omega(x) = \begin{cases}
        0  & \text{if}\, \dis{(x,\p \omega)}<\frac{\delta}{2},\\
             1  & \text{if} \,\,\dis{(x,\p\omega)}>\delta.
     \end{cases}
\]

In Appendix~\ref{J_coerciv} we will show that under the Hypothesis~\ref{hypo} and given any cut-off function satisfying (i)--(iii) above, the two main assumptions of Theorem~\ref{smoothness} are satisfied for $s=3$. Therefore, this theorem applies to solve the null controllability problem \eqref{pf}--\eqref{dual} with the additional smoothness property that $$u_* \in \mathscr{X}_{4}(\M) \quad\text{and}\quad U_* \in \mathscr{X}_3(\M),$$ if the initial datum $(g_0,g_1)$ belongs to the space $\mathcal D((-\Delta)^3)$. This smoothness class for the continuum solution to \eqref{pf}--\eqref{dual} will be important in our numerical analysis. 

Before closing the section, let us emphasize that the results in this paper can also be applied to the problem of (interior) exact controllability, where the final state $(u(T,x),\p_t u(T,x))$ can be any pair of functions $(h_0,h_1) \in \mathcal D((-\Delta)^{3})$. This is a consequence of the equivalence of the null and exact controllability problems for the wave equation. To illustrate this equivalence, let $u_1$ denote the unique solution to equation \eqref{pf} with a homogeneous source term $U=0$, but with the difference that the initial conditions are imposed at the final time $t=T$ that is to say $u_1(T,\cdot)=h_0$ and $\p_t u_1(T,\cdot)=h_1$. This is possible due to the time-reversibility of the wave equation. Subsequently, let $(u_1|_{t=0},\p_t u_1|_{t=0})=(\tilde{g}_0,\tilde{g}_1)$. Finally, let $\chi_\omega U$ be a null control that drives the system from initial data $(g_0-\tilde{g}_0,g_1-\tilde{g}_1)$ to $(0,0)$. It is clear that $\chi_\omega U$ is a control that drives the solution $u$ to equation \eqref{pf} from $(g_0,g_1)$ to $(h_0,h_1)$.



\subsection{Previous literature}

It is well-known that methods based on minimizing discrete analogues
of the Lagrangian \eqref{cont-Lagr} may fail to converge. This is the case, for
example, when second-order central finite differences in both time and
space are used to discretize \eqref{dual}, and the so obtained discretization
of \eqref{cont-Lagr} is minimized by using the conjugate gradient algorithm. This was first observed by Glowinski et al. in a series of works in early 1990s. An excellent summary of these works is provided by Glowinski and Lions in Sections 6.8--9 of \cite{GL}. 
It was observed that trouble lies with the high-frequency components of the discrete solution, see e.g. Section 6.8.6 of \cite{GL}, and different regularization techniques were proposed. For example, a Tikhonov type regularization procedure based on a use of the biharmonic operator is discussed in detail in \cite{GL}, and the efficiency of the regularization is demonstrated by numerical experiments.

The spurious modes arising at high frequencies from a finite-difference semi-discretization of the one dimensional wave equation were first rigorously analyzed in \cite{IZ}. In particular, it was shown that the analogue of the estimate in Theorem \ref{observability} fails on the discrete level.
Several numerical methods based on filtering of the spurious high frequency modes were subsequently proposed. As an early example of a result in this tradition, we mention \cite{M}  where weak convergence of a subsequence of semi-discrete approximations of a control function for the one dimensional wave equation was proven.
For a thorough review of the filtering approach, we refer to the monograph \cite{EZ}. 
There it is also shown that a semi-discrete variant of the approach has optimal convergence under the assumption that the analogue of the estimate in Theorem \ref{observability} is recovered on the discrete level after suitable filtering. 
However, it is not known if such filtered estimates hold in general, when only the geometric control condition is assumed, see the discussion in Section 5.3 of \cite{EZ}.

Instead of considering the control function satisfying (\ref{dual}), it is also possible to follow Russell's {\em stabilization implies control} principle \cite{Russell}. On the continuum level, this involves an alternating iteration solving  forward and backward wave equations. A suitably semi-discretized version of this scheme leads to a solution method to the null control problem with a rate of convergence exhibiting only a logarithmic loss when compared to the optimal rate \cite{CMT}. 
However, the scheme requires that the alternating iteration is stopped after a specific number of steps, depending for example on the constant $C$ in Theorem \ref{observability}, and this stopping criterion may not be easy to implement in practice. 
As demonstrated in Section 1.7.1.2 of \cite{EZ}, the iteration in fact diverges as the number of steps grows too large.

In a recent work \cite{CiMu_control}, M\"unch et al. formulate the controllability problem so that the wave equation (\ref{dual}) enters into the Lagrangian functional (\ref{cont-Lagr}) via a Lagrange multiplier. The Lagrangian functional is further augmented with the $L^2$-norm of $\Box U$.
In a subsequent work \cite{MoMu}, a Lagrange multiplier is used to impose the wave equation as a first order system. The efficiency of the resulting methods is demonstrated by numerical experiments, however, their convergence analysis is not complete as it is not known if the discrete inf-sup constants for the methods, see (39) and (6.9) in \cite{CiMu_control} and in \cite{MoMu}, respectively, are uniformly bounded from below.

Our approach is based on a Lagrangian functional where the initial conditions in \eqref{pf} together with the final conditions in \eqref{ctrl} are imposed via penalty terms, and similarly to \cite{CiMu_control}, the equations \eqref{pf} and \eqref{dual} are imposed via Lagrange multipliers.
Instead of augmentation, we add Tikhonov type regularization terms that vanish at the correct rate as the mesh size tends to zero. This allows us to prove a discrete inf-sup property (Proposition \ref{estimate}), and subsequently an optimal convergence rate (Theorem \ref{t1}). 
The present method can be seen as the continuation of our previous work in \cite{BFO}, where we studied numerical approximation of the dual problem to the controllability problem discussed here, that is, the data assimilation problem subject to the wave equation. A detailed comparison between these two works is given in Section~\ref{conc_remarks}.  

\subsection{Outline of the paper}
We start Section~\ref{discretization_sec} with presenting the first order finite element spaces that are used for the numerical approximation of the null controllability problem. A Lagrangian functional is then formulated in the discrete level, and the main theorem is stated (Theorem~\ref{t1}) that gives a numerical method for solving the null controllability problem. Section~\ref{infsup_sec} is concerned with proving a suitable inf-sup stability estimate (Proposition~\ref{estimate}) for the discrete Lagrangian. We also show the existence of a unique critical point for the discrete Lagrangian. In Section~\ref{weak_sec}, the inf-sup stability estimate is used together with a continuity estimate for the residual error (Lemma~\ref{Aupperbound}) to obtain a weak a priori control on the error function (Proposition~\ref{discretesol}). This proposition is then used to obtain an approximate version of the observability estimate in Theorem~\ref{observability} at the discrete level (Proposition~\ref{discreteobs}). Section~\ref{strong_est_section} is concerned with the proof of Theorem~\ref{t1}. There, the key ingredients are the coercivity Lemma~\ref{positivityf}, together with the approximate discrete observability estimate (Proposition~\ref{discreteobs}). Finally, in Section~\ref{conc_remarks} we give a detailed comparison with our earlier work for the data assimilation problem together along with some concluding remarks.

\section*{Acknowledgment} 
EB was supported by EPSRC grants EP/P01576X/1 and EP/P012434/1, AF by EPSRC grant EP/P01593X/1 and LO by EPSRC grants EP/L026473/1 and EP/P01593X/1. The authors would like to thank the anonymous reviewers whose comments helped to improve and clarify this article.
\section{Discretization}
\label{discretization_sec}
Let us now present the discretization approach for \eqref{pf}--\eqref{dual}. We will use finite differences in time and first order finite elements in space. Let $N \in \mathbb{N}$ and define $\tau=\frac{T}{N}$ to denote the uniform length of the time-steps in our numerical method. Also, let $\{t_k\}_{k=0}^N$ be defined through $t_k=k\tau$. We begin by discretizing the boundary $\pd \Omega$ and denote the resultant polyhedral domain by $\Omega_h$. This polyhedral approximation is assumed to be sufficiently close to $\Omega$ in the sense that
\bel{boundarydisc}
\dis{(x,\pd \Omega)}\leq C\, h^2,\quad \forall x \in \p\Omega_h,
\ee
for some constant $C>0$ that is independent of $h$. This is always possible since $\Omega$ has a smooth boundary (see \cite{BK} for example). Subsequently, we consider a spatial mesh $\mathcal{T}_h$ which is a conforming quasi uniform triangulation of the polyhedral domain $\Omega_h$ and define $h_{K}$ to be the local space mesh size. We set $h=\max_{K \in \mathcal{T}_h} h_{K}$ to be the global mesh parameter in space and make the standing assumption that the discrete time steps $\tau$ and the spatial mesh parameter $h$ are comparable in size, that is to say $\tau=\mathcal{O}(h).$

We now define the spatial finite element space $\mathbb V_h$ to be the space of piece-wise affine continuous finite elements satisfying zero boundary condition,
$$ \mathbb V_h = \{ v \in H^1_0(\Omega_h)\,:\, v|_K \in \mathbb{P}_1(K),\, \forall K \in \mathcal{T}_h\}.$$  
For each $u,v \in \mathbb V_h$, let 
$$(u,v)_h=\int_{\Omega_h} u(x)v(x) \,dx \quad \text{and} \quad a_h(u,v)=\int_{\Omega_h} \nabla u(x) \cdot \nabla v(x) \,dx,$$
and 
$$\|u\|_h=\sqrt{(u,u)_h}.$$ 
Next, we define the space-time mesh $\mathbb V_h^{N+1}=\underbrace{\mathbb V_h \times \mathbb V_h\times\ldots\times \mathbb V_h}_{\text{$N+1$ times}}$ and subsequently for each $$u=(u^0,u^1,\ldots,u^N) \in \mathbb V_h^{N+1},$$ we define the backward and forward discrete time differences $\p_\tau,\tilde{\p}_\tau$ as follows: 

$$
\begin{aligned}
\partial_{\tau} u^n &= \frac{u^n - u^{n-1}}{\tau}\quad \text{for}\quad n=1,\ldots,N, \\
\tilde{\p}_\tau u^n &= \frac{u^n-u^{n+1}}{\tau}\quad \text{for}\quad n=0,\ldots,N-1.
\end{aligned}
$$

\noindent We note in passing that the forward discrete time difference $\tilde{\p}_\tau $ acting on $(u^n)_{n=0}^N$ can be thought as the backward time difference for the discrete function $(u^{N-n})_{n=0}^N$. Note also that second order time differences can be written through
$$
\begin{aligned}
\partial^2_{\tau} u^n &= \frac{u^{n} -2 u^{n-1}+u^{n-2}}{\tau^2}\quad \text{for}\quad n=2,\ldots,N, \\
\tilde{\p}^2_\tau u^n &= \frac{u^n-2u^{n+1}+u^{n+2}}{\tau^2}\quad \text{for}\quad n=0,\ldots,N-2.
\end{aligned}
$$

Finally, consider any smooth extension of $\chi_\omega$ to $\Omega \cup \Omega_h$ and define a non-negative discrete approximate $\chi_h$ of the smooth function $\chi_\omega$ such that $\chi_h \in \mathbb V_h$ and 
\begin{equation}\label{eq:disc_chi}
\|\chi_\omega - \chi_h\|_{L^\infty(\Omega_h)} + h
\|\chi_\omega - \chi_h\|_{W^{1,\infty}(\Omega_h)} \leq C\,h^2,
\end{equation}
where $C>0$ is independent of $h$. Note that this is possible due to the smoothness assumption on $\chi_\omega$ along with smoothness of $\p\Omega$ and equation \eqref{boundarydisc}.

We now return to the null controllability problem \eqref{pf}--\eqref{dual}. Given any $u=(u^0,\ldots,u^N)$, $U=(U^0,\ldots,U^N)$ in $\mathbb V_h^{N+1}$ and $z=(z^2,\ldots,z^N)$, $Z=(Z^0,\ldots,Z^{N-2})$ in $\mathbb V_h^{N-1}$, we define the discrete Lagrangian functional 
$$\mathcal{J}: \mathbb V_h^{N+1}\times \mathbb V_h^{N+1}\times\mathbb V_{h}^{N-1}\times \mathbb V_h^{N-1} \to \mathbb{R}$$ 
through the expression

\begin{equation} \label{Lagrangian}
\begin{aligned}
\mathcal{J}(u,U,z,Z)=&\, \mathcal{J}_{0}(u,U,z,Z)+\mathcal{J}_{1}(U),\\
\mathcal{J}_{0}(u,U,z,Z)=&\, \mathcal G(u,z)-\tau \sum_{n=2}^N (\chi_hU,z)_h+\mathcal G^*(Z,U)+\mathcal R(u),\\
\mathcal{J}_1(U)=&\, \frac{h^2}{2}\|\nabla U^N\|_h^2+\frac{h^2}{2}\|\p_\tau \nabla U^N\|_h^2+\frac{h^2}{2}\|\p_\tau \nabla U^1\|_h^2+\frac{\tau h^2}{2} \sum_{n=1}^N\|\p_\tau\nabla U^n\|_h^2,
\end{aligned}
\end{equation}
\[
\begin{aligned}
\mathcal G(u,z)&= \tau\sum_{n=2}^{N}\left((\p^2_\tau u^n,z^n)_h+a_h(u^n,z^n)\right),\\
\mathcal G^*(Z,U)&=\tau \sum_{n=0}^{N-2}\left((Z^n,\tilde{\p}^2_\tau U^n)_h+a_h(Z^n,U^n)\right),\\
\mathcal R(u)&=\frac{1}{2}\left[\|\nabla u^N\|_h^2+\|\p_\tau u^N\|_h^2+\|\nabla (u^0-g_0)\|_h^2+\|\p_\tau u^1-g_1\|_h^2\right].
\end{aligned}
\bigskip
\]

Let us make some remarks about the discrete Lagrangian $\mathcal J$. Here, the variables $u$ and $U$ should be interpreted as discrete analogoues of the state variable and the control function, while $z$ and $Z$ are discrete variables. The terms $\mathcal G(u,z) - \tau \sum_{n=2}^N (\chi_\omega U, z)_h$ and $\mathcal G^*(Z,U)$ are weak formulations of the first
equations in \eqref{pf} and in \eqref{dual}, respectively. Although in continuum, the forward and backward wave equations are completely equivalent, we use a backward discrete wave equation for the discrete control variable $U$. This will be important in the proof of convergence rates for our numerical analysis (see Lemma~\ref{mixedterm}). The functional $\mathcal R$ imposes the initial conditions in \eqref{pf} as well as the final conditions \eqref{ctrl}. Note that the initial states $z^0$, $z^1$ for $z$ and final states $Z^{N-1}$, $Z^{N}$ for $Z$ do not appear in the formulation and can be taken to be zero. Intuitively, the Lagrange multipliers are solving in-homogeneous wave equations with zero initial or final data. To summarize, $\mathcal J_0$ corresponds to equations \eqref{pf}--\eqref{dual}.

We have incorporated the numerical stabilizers (also called regularizers) in the discrete level through the functional $\mathcal J_1(U)$. The design of these terms is driven with the goal of minimizing the errors in the numerical approximation of the null controllability problem. The first two terms in $\mathcal J_1$ correspond to the energy for the wave equation \eqref{dual} at time $t=T$ and seem a natural inclusion, while the remaining two terms are in part motivated by our previous works for data assimilation problems for heat and wave equations \cite{BFO,BIO}. The regularization term in mixed derivatives also appears in \cite{BE}. 

Heuristically, we expect to have a critical (saddle) point in the sense that the Lagrangian attains the value 
$$ \inf_{(u,U)\in \mathbb V_h^{2N+2}} \sup_{(z,Z)\in \mathbb V_h^{2N-2}} \mathcal J(u,U,z,Z).$$  
Moreover, we expect this critical point to converge to the continuum solution of the control problem \eqref{pf}--\eqref{dual} with the Lagrange multipliers $(z,Z)$ converging to zero as $h \to 0$. 

The Euler-Lagrange equations for the functional $\mathcal J$ can be written in the form 
$$ \langle D_u \mathcal J , v\rangle +\langle D_U \mathcal J , V\rangle+ \langle D_z \mathcal J, w\rangle+\langle D_Z \mathcal J , W\rangle=0\quad \forall (v,V,w,W) \in \mathbb V_h^{4N},$$
where $D_s$ denotes the Fr\'{e}chet derivative with respect to $s \in \{u,U,z,Z\}$. Letting $$x=(u,U,z,Z) \in \mathbb V_h^{4N} \quad\text{and}\quad y=(v,V,w,W) \in \mathbb{V}_h^{4N},$$ we see that the Euler-Lagrange equations can be recast in the form 
\bel{EL}
\mathcal{A}(x;y)=a_h(v^0,g_0)+(\p_\tau v^1,g_1)_h \quad \text{for all} \quad y \in \mathbb V_h^{4N},
\ee
where $\mathcal{A}:\mathbb V_h^{4N} \times \mathbb V_h^{4N} \to \R$ is a bi-linear form defined through 
\[
\mathcal A(x;y)=\mathcal A_0(u,z;v)+\mathcal A_1(U,z,Z;V)+\mathcal A_2(u,U;w)+\mathcal A_3(U;W),
\]
with
\[
\begin{aligned}
\mathcal A_0(u,z;v)=&\,\mathcal G(v,z)+a_h(u^0,v^0)+a_h(u^N,v^N)+(\p_\tau u^1,\p_\tau v^1)_h+(\p_\tau u^N,\p_\tau v^N)_h,\\
\mathcal A_1(U,z,Z;V)=&\,\mathcal G^*(Z,V)-\tau \sum_{n=2}^N(\chi_hV,z)_h+h^2(\nabla U^N,\nabla V^N)_h+h^2(\p_\tau \nabla U^N,\p_\tau \nabla V^N)_h\\
&+h^2(\p_\tau \nabla U^1,\p_\tau \nabla V^1)_h+\tau \sum_{n=1}^N(h\, \p_\tau\nabla U^n,h\,\p_\tau\nabla V^n)_h,\\
\mathcal A_2(u,U;w)=&\,\mathcal G(u,w)-\tau\sum_{n=2}^N(\chi_hU,w)_h,\\
\mathcal A_3(U;W)=&\,\mathcal G^*(W,U).
\end{aligned}
\]

Observe, in particular that the expressions for $\mathcal A_2$ and $\mathcal A_3$ imply that the Euler-Lagrange equations for $u$ and $U$ enforce discrete versions of \eqref{pf} and \eqref{dual}. Indeed, the state variable $u$ must solve the discrete forward wave equation with source term $\chi_h U$, while the control variable $U$ must solve the discrete backward wave equation. We are now ready to state the main theorem in the paper as follows.
\bigskip

\begin{thm}
\label{t1}
Suppose that Hypothesis~\ref{hypo} holds for the set $\mathcal O=(0,T)\times\omega$. Let $\chi_\omega$ be any function that satisfies properties \em{(i)--(iii)} in \eqref{cutoff_conditions}. Let $(g_0,g_1) \in \mathcal D((-\Delta)^{3})$ and denote by $(u_*,U_*)$, the unique continuum solution to the interior null controllability problem \eqref{pf}--\eqref{dual}. Then, there exists $h_0>0$, such that for all $0<h<h_0$, the Euler-Lagrange equation \eqref{EL} admits a unique solution denoted by $(u_h,U_h,z_h,Z_h)$. Moreover, for $n=1,\ldots,N$:
\begin{itemize}
\item[(i)]$\|U_*^n - U^n_h\|_{L^2(\Omega)}+\|\p_t U_*^n -\p_\tau U^n_h\|_{H^{-1}(\Omega)}\leq C\, h \|(g_0,g_1)\|_{H^{4}(\Omega)\times H^3(\Omega)},$
\item[(ii)]$\|u_*^n - u^{n}_h\|_{H^1(\Omega)}+\|\p_t u_*^n-\p_\tau u^n_h\|_{L^{2}(\Omega)} \leq C\,h \|(g_0,g_1)\|_{H^{4}(\Omega)\times H^3(\Omega)},$
\end{itemize}
where $U_*^n(\cdot)=U_*(n\tau,\cdot)$, $u_*^n(\cdot)=u_*(n\tau,\cdot)$ and $C>0$ is a constant independent of the mesh parameter $h$\,\footnote{\,Recall that $h$ and $\tau$ are assumed to be comparable, that is $\tau=\mathcal O(h)$.} and only depends on $T$, $\M$, $\omega$, $\delta$.\footnote{\,See Proposition~\ref{discretesol} for the estimates of $z_h$ and $Z_h$.}
\end{thm}

\section{Inf-sup stability estimate}
\label{infsup_sec}
This section is concerned with the study of the Euler-Lagrange
equation \eqref{EL}. 
First, we define $\kappa,\tilde{\kappa}>0$ to be constants such that for all $u \in \mathbb V_h$
\bel{reversepoincare}
\max{\{h,\tau\}} \|\nabla u\|_h \leq \kappa\|u\|_h\quad \text{and} \quad \tilde{\kappa}\|u\|_h\leq \|\nabla u\|_h,
\ee
where we recall that $h,\tau$ are the mesh parameters. The existence of these constants is guaranteed by the discrete inverse
inequality that follows from the fact that the space-mesh is quasi-uniform (see for instance \cite[Lemma 4.5.3]{BS08}) together with the Poincar\'{e} inequality and the standing assumption that $\tau=\mathcal O(h)$.

We introduce the following discrete norms and semi-norms:
\[
\begin{split}
\tnorm (u,U) \tnorm_R^2 =&\,\|\nabla u^{N}\|_h^2+\|\p_\tau u^N\|_h^2+\|\nabla u^0\|_h^2+\|\p_\tau u^1\|_h^2+\tau\sum_{n=1}^N\|h\, \p_\tau \nabla U^n\|_h^2\\
&+h^2\|\nabla \p_\tau U^N\|_h^2+h^2\|\p_\tau \nabla U^1\|_h^2+h^2\|\nabla U^N\|_h^2,\\ 
\tnorm u \tnorm_F^2 =&\, \tau \sum_{n=1}^N(\|\partial_\tau u^n\|_h^2 + \|\nabla u^n\|_h^2),\\
\tnorm U \tnorm_{F'}^2 =&\, \tau \sum_{n=2}^N(\|\partial_{\tau}^2 U^n\|_h^2+\|\p_\tau\nabla U^n\|^2_h+\|\partial_\tau U^n\|_h^2 + \|\nabla U^n\|_h^2)+\|\p_\tau U^1\|_h^2+\|\nabla U^1\|_h^2,\\
\tnorm z \tnorm_{D}^2 =&\, \tau\sum_{n=2}^N \|z^n\|_h^2 +\tau\sum_{n=2}^N \|\nabla \mathcal{I}{z}^n\|_h^2+ \|\nabla \mathcal{I}{z}^N\|_h^2+\|z^N\|_h^2\\
\tnorm Z \tnorm_{D'}^2=&\,\tau\sum_{n=0}^{N-2} \|Z^n\|_h^2 +\tau\sum_{n=0}^{N-2} \|\nabla \tilde{\mathcal{I}}{Z}^n\|_h^2+ \|\nabla \tilde{\mathcal{I}}{Z}^0\|_h^2+\|Z^0\|_h^2\\
\tnorm (u,U,z,Z) \tnorm_C^2=&\,\tnorm (u,U) \tnorm_R^2+\tau\sum_{n=2}^N\|z^n\|_h^2+\tau \sum_{n=0}^{N-2}\|Z^n\|_h^2.
\end{split}
\]
Here, 
$$\mathcal{I}{z}^n = \tau \sum_{m=0}^n (1+m\tau) z^m\quad \text{and}\quad \tilde{\mathcal{I}}{Z}^n= \tau \sum_{m=n}^N (1+(N-m)\tau) Z^m$$ 
where we have defined $z^0=z^1=0$ and $Z^{N-1}=Z^N=0$. Note that using the Poincar\'{e} inequality we have the following:
$$ \|\nabla \mathcal{I}{z}^n\|_h \geq C\, \|\mathcal{I}{z}^n\|_h \hspace{5mm} n=2,...,N,$$
for some $C>0$ independent of $h$, with an analogous estimate holding for $\tilde{\mathcal{I}}Z$ as well. 

The above norms and semi-norms have the following interpretations. The $\tnorm (\cdot,\cdot) \tnorm_R$ semi-norm captures the stability properties of the bi-linear form $\mathcal A$ due to the
regularization terms in $\mathcal J_1$ and the data fitting terms in $\mathcal R$. The norms $\tnorm\cdot\tnorm_F$, $\tnorm\cdot\tnorm_{F'}$, $\tnorm\cdot\tnorm_{D}$ and $\tnorm\cdot\tnorm_{D'}$ quantify stability
properties of the discrete wave equations for $u$, $U$, $z$ and $Z$ given by
$\mathcal G(u,z)$ and $\mathcal G^*(Z,U)$. There is a delicate counter balance in the strength of the norms $\tnorm\cdot\tnorm_{F}, \tnorm\cdot\tnorm_{F'}$ (in terms of the Sobolev scales) for the functions $u, U$ compared to that of $\tnorm\cdot\tnorm_{D}, \tnorm\cdot\tnorm_{D'}$ for the Lagrange multipliers $z, Z$. For instance, in the continuum limit $\tau,h \to 0$, the $\tnorm\cdot\tnorm_{F}$ norm is reminiscent to $\|\cdot\|_{H^1(\M)}$, while the $\tnorm\cdot\tnorm_D$ norm is analogous to $\|\cdot\|_{H^{-1}(0,T;H^1(\Omega))}+\|\cdot\|_{L^2(\M)}$. The $\tnorm(\cdot,\cdot)\tnorm_{C}$ semi-norm quantifies continuity of $\mathcal A$ in the dual
variables from above and below, see Proposition~\ref{estimate} below and Lemma~\ref{Aupperbound}.

The rest of this section is concerned with the following proposition.

\begin{prop}
\label{estimate}
There exists $h_0,C>0$ such that for all $h\in(0,h_0)$ and all $x=(u,U,z,Z) \in \mathbb V_h^{4N}$ there exists $y=(v,V,w,W) \in \mathbb V_h^{4N}$ satisfying:
$$ \mathcal A(x;y)\gtrsim\tnorm(u,U)\tnorm_R^2 +h^2\tnorm u \tnorm_{F}^2 + h^2\tnorm U \tnorm_{F'}^2+\tnorm z \tnorm_{D}^2+\tnorm Z \tnorm_{D'}^2 \gtrsim \tnorm y\tnorm_C^2.$$
\end{prop}

\begin{rmk}
Throughout the remainder of the paper we use the notation $A \gtrsim B$, to imply the existence of a positive constant $C>0$ independent of the mesh parameter $h$, such that $A\geq C B$.
\end{rmk}

We can use Proposition~\ref{estimate} to show that the Euler-Lagrange equation \eqref{EL} admits a unique solution $x_h=(u_h,U_h,z_h,Z_h) \in \mathbb V_h^{4N}$. Indeed, let $N_h$ denote the dimension of $\mathbb V_h$. Equation \eqref{EL} is a linear system governed by a $4N_hN \times 4N_hN$ matrix. Existence and uniqueness of a discrete solution $x_h$ will follow, if we can show that the kernel of this matrix is trivial. But this follows immediately from Proposition~\ref{estimate}.  

Before presenting the proof of Proposition~\ref{estimate}, we state a
few lemmas, the first of which is trivial. 
\begin{lem}
\label{algebra}
Let $x=(u,U,z,Z)\in \mathbb V_h^{4N}$. If $y=(u,U,-z,-Z)$, then 
$$\mathcal{A}(x;y) =\tnorm(u,U)\tnorm_R^2.$$
\end{lem}
The estimates in the next lemma are discrete analogues of energy estimates for the wave equation corresponding to various Sobolev norms. The energy estimates for $u$ and $U$ will be stronger in the Sobolev scale but eventually rescaled by $h^2$ and this will be balanced by weaker Sobolev spaces with no scaling on the dual variables $z,Z$. For the proof, we refer the reader to \cite[Remark 1, Lemma 3.4, Lemma 3.5]{BFO}.
\begin{lem}
\label{energyuUzZ}
Let $(u,U,z,Z) \in \mathbb V_h^{4N}$. We define $z^0=z^1=Z^{N-1}=Z^N=0$. Define the test function $y=(v,V,w,W)$ through
\[
\begin{aligned}
v^n&= \mathcal I z^n \quad \text{for}\quad n=0,\ldots,N,\\
V^n&=\tilde{\mathcal{I}} Z^n \quad \text{for}\quad n=0,\ldots,N,\\
w^n&=(2T-n\tau)\p_\tau u^n\quad \text{for}\quad n=2,\ldots,N,\\
W^n&=\tilde{\p}^2_\tau U^n+(2T-(N-n)\tau)\tilde{\p}_\tau U^n\quad \text{for}\quad n=0,\ldots,N-2.
\end{aligned}
\]
The following estimates hold:
\bel{energyest}
\begin{aligned}
&\mathcal G(u,w) \gtrsim \tnorm u \tnorm_{F}^2-\tnorm(u,0)\tnorm_R^2,\quad &\mathcal G(v,z)&\gtrsim \tnorm z \tnorm_{D}^2,\\
&\mathcal G^*(h^2W,U) \gtrsim h^2\tnorm U \tnorm_{F'}^2-\tnorm(0,U)\tnorm_R^2,\quad &\mathcal G^*(Z,V)&\gtrsim \tnorm Z \tnorm_{D'}^2,
\end{aligned}
\ee
where the constants in the inequalities only depend on $T,\Omega$.
\end{lem}

We are now ready to prove Proposition~\ref{estimate}.

\begin{proof}[Proof of Proposition~\ref{estimate}]
Let $x=(u,U,z,Z)$ and define $y=(\hat{v},\hat{V},\hat{w},\hat{W}) \in \mathbb{V}^{4N}$ through
\bel{testfcns}
\begin{aligned}
\hat{v}^n = u^n+\gamma v^n,   \quad  &\hat{V}^n=U^n+\alpha V^n,\\
\hat{w}^n = -z^n+\alpha h^2 w^n,  \quad     &\hat{W}^n=-Z^n+\gamma h^2 W^n,
\end{aligned}
\ee
where $\gamma>\alpha>0$ and $v,V,w,W$ are chosen as in Lemma~\ref{energyuUzZ}. Recalling the definition of the linear form $\mathcal A(x;y)$ together with Lemma~\ref{algebra}, we write
\[
\mathcal A(x;y)=\tnorm(u,U)\tnorm_R^2+\gamma \mathcal A_0(u,z;v)+\alpha \mathcal A_1(U,z,Z;V)+\alpha h^2 \mathcal A_2(u,U;w)+\gamma h^2 \mathcal A_3(U;W).
\]
By Lemma~\ref{energyuUzZ}, there exists $C_1,C_2>0$ only depending on $T,\Omega$ such that
\bel{energycomb}
\begin{aligned}
&\alpha h^2\mathcal G(u,w)+\gamma \mathcal G(v,z)+\alpha \mathcal G(Z,V)+\gamma h^2\mathcal G(W,U) \geq\,\\
&C_1(h^2\gamma\tnorm U \tnorm_{F'}^2+\alpha h^2\tnorm u \tnorm_F^2+\gamma\tnorm z \tnorm_{D}^2+\alpha\tnorm Z \tnorm_{D'}^2)-C_2\gamma\tnorm(u,U)\tnorm_R^2.
\end{aligned}
\ee
We now set $\alpha=\alpha_0\gamma$ for a fixed $0<\alpha_0<\min{\{1,\frac{3}{4}\tilde{\kappa}^2C_1^2,\frac{C_1^2}{4T^2}\}}$ and show that the proposition holds for this choice of $y \in \mathbb V_h^{4N}$ when $\gamma$ is sufficiently small independent of $h$. 
First, note that
$$ a_h(u^0,v^0)=a_h(\p_\tau u^1,\p_\tau v^1)=a_h(U^N,V^N)=a_h(\p_\tau U^N,\p_\tau V^N)=0.$$
We use the Cauchy-Schwarz inequality to obtain the following bounds for the remaining (possibly) negative terms in $\mathcal A(x;y)$
\[
\begin{aligned}
|a_h(u^N,v^N)|&\leq \frac{1}{C_1} \|\nabla u^N\|_h^2+\frac{C_1}{4}\|\nabla \mathcal{I} z^N\|_h^2,\\
|(\p_\tau u^N,\p_\tau v^N)_h|&\leq\frac{(1+T)^2}{C_1}\|\p_\tau u^N\|_h^2+\frac{C_1}{4}\|z^N\|_h^2,\\
h^2|(\p_\tau \nabla U^1,\p_\tau \nabla V^1)_h|&\leq h^2\frac{\kappa^2(1+T)^2}{C_1}\|\p_\tau \nabla U^1\|_h^2+\frac{C_1}{4}\|Z^0\|_h^2,\\ 
\tau \sum_{n=1}^N |(\tau \p_\tau \nabla U^n,\tau \p_\tau \nabla V^n)_h| &\leq \frac{\kappa^2(1+T)^2}{C_1} \tau \sum_{n=1}^N\|\tau \p_\tau \nabla U^n\|_h^2+\frac{C_1}{4}\tau \sum_{n=0}^{N-2}\|Z^n\|_h^2,\\
h^2|\tau \sum_{n=2}^N (\chi_h U^n,w^n)_h| &\leq \frac{4T^2}{C_1}\tau\sum_{n=2}^Nh^2\|U^n\|_h^2+\frac{C_1}{4}\tau \sum_{n=2}^Nh^2\|\p_\tau u^n\|_h^2,\\
|\tau\sum_{n=2}^N (\chi_h V,z)_h| &\leq \frac{1}{\tilde{\kappa}^2C_1}\tau \sum_{n=2}^N\|z^n\|_h^2+\frac{C_1}{4}\tau \sum_{n=0}^{N-2}\|\nabla\tilde{\mathcal{I}} Z^n\|_h^2.
\end{aligned}
\]
Combining these bounds we deduce that
\[
\begin{aligned}
\mathcal A(x;y)\geq\,& \tnorm(u,U)\tnorm_R^2-C_3\gamma\tnorm(u,U)\tnorm_R^2+\alpha\frac{3C_1}{4}\tnorm Z \tnorm_{D'}^2+\alpha\frac{3C_1}{4}h^2\tnorm u \tnorm_F^2\\
&+(\frac{3C_1\gamma}{4}-\frac{\alpha}{\tilde{\kappa}^2C_1})\tnorm z \tnorm_D^2+(C_1\gamma-\frac{4T^2\alpha}{C_1})h^2\tnorm U \tnorm_{F'}^2,
\end{aligned}
\]
where $C_3=C_2+2(1+\kappa^2)(1+T)^2\frac{1}{C_1}$. The first claimed inequality then follows for $\gamma$ sufficiently small. To prove the second inequality in the proposition, we use the Cauchy-Schwarz inequality to obtain the following bounds for $y \in \mathbb V_h^{4N}$:
\bel{coerc2}
\begin{aligned}
&\tnorm(\hat{v},0)\tnorm_R^2 \lesssim \tnorm(u,0)\tnorm_R^2+\tnorm z \tnorm_D^2,\quad \tnorm(0,\hat{V})\tnorm_R^2 \lesssim \tnorm(0,U)\tnorm_R^2+\tnorm Z \tnorm_{D'}^2,\\
&\tau \sum_{n=2}^N\|\hat{w}^n\|_h^2 \lesssim \tau \sum_{n=2}^N\|z^n\|_h^2+h^2\tnorm u \tnorm_F^2,\quad
\tau \sum_{n=0}^{N-2}\|\hat{W}^n\|_h^2\lesssim \tau \sum_{n=2}^N \|Z^n\|_h^2+h^2\tnorm U \tnorm_{F'}^2.
\end{aligned}
\ee
\end{proof}


\section{A weak a priori error estimate and an approximate discrete observability estimate for the error}
\label{weak_sec}
Throughout this section, we will let $(u_h,U_h,z_h,Z_h)$ denote the unique solution to equation \eqref{EL}. The main goal here is to prove an approximate discrete analogue of the continuum observability estimate\footnote{\,Not to be confused with discrete observability estimates in the literature.}. This will be done in several steps. We start by proving a weak preliminary error estimate (Proposition~\ref{discretesol}), and then use this estimate to prove Proposition~\ref{discreteobs} that we call an approximate discrete observability estimate for the error function. This proposition will subsequently be used as a key ingredient to prove the main theorem.

In what follows, we will let $(u_*,U_*)$ denote the continuum solution to equations \eqref{pf}$-$\eqref{dual}. Let us observe that since $(g_0,g_1) \in \mathcal D((-\Delta)^3)$, it follows from Theorem~\ref{smoothness} that
$$ \|u_*\|_{\mathscr X_4(\M)}+\|U_*\|_ {\mathscr X_3(\M)} \leq C \|(g_0,g_1)\|_{H^4(\Omega)\times H^3(\Omega)}$$
where $C>0$ only depends on $\M,\omega,\delta$ and $\mathscr X_3(\M)=\bigcap_{k=0}^{3} \mathcal C^k(0,T;H^{3-k}(\Omega)).$

Due to the mismatch between $\Omega$ and $\Omega_h$, we extend the functions in $\mathbb V_h$ to all of $\Omega_h \cup \Omega$ by setting them to zero on the set $\Omega\setminus \Omega_h$. We will also utilize the extension operator (see \cite{Stein}) $E:H^s(\Omega) \to H^s(\Omega\cup\Omega_h)$, $s\geq0$ to define the extended functions $u^e_*$ and $U_*^e$ on the set $(0,T)\times (\Omega\cup \Omega_h)$ through
$$ u^e_*(t,\cdot)=Eu_*(t,\cdot)\quad\text{and}\quad U^e_*(t,\cdot)=EU_*(t,\cdot)\quad t \in [0,T].$$
We will slightly abuse the notation by dropping the superscript $e$ when there is no confusion.

Let us recall the definition of the $H^1$ projection interpolator $\pi_h: H^1_0(\Omega)\to\mathbb V_h$ defined through
\bel{H_1proj}
a_h(\pi_h u, v) = a_h(u,v) \quad \forall\, v \in \mathbb V_h.
\ee
We have the following lemma. For the proof, we refer the reader to \cite[Lemma 4.2]{BFO}.
\begin{lem}
\label{interpolator}
Let $u \in H^1_0(\Omega)$. Then
\[
\|Eu-\pi_h u\|_{L^2(\Omega_h)} \lesssim h \|u\|_{H^1_0(\Omega)},\]
and if additionally $ u \in H^2(\Omega)$, then
\[
\|Eu-\pi_h u\|_{H^1(\Omega_h)} \lesssim  h\|u\|_{H^2(\Omega)}.
\]
\end{lem}
We define 
\bel{x}
 \tilde{u}_h=u_h-\pi_h u_*, \quad \tilde{U}_h=U_h-\pi_h U_*,\quad \text{and} \quad x_h=(\tilde{u}_h,\tilde{U}_h,z_h,Z_h).
\ee
We have the following proposition.
\begin{prop}
\label{discretesol}
There exists $h_0>0$ such that for all $0<h<h_0$, the following estimate holds
$$ \tnorm z_h \tnorm_{D}+\tnorm Z_h \tnorm_{D'}+h\tnorm \tilde{u}_h \tnorm_F+h\tnorm \tilde{U}_h \tnorm_{F'}+\tnorm(\tilde{u}_h,\tilde{U}_h)\tnorm_R \lesssim h\|(g_0,g_1)\|_{H^{4}(\Omega)\times H^3(\Omega)}.$$
\end{prop} 

Let us remark that as an immediate consequence of this proposition, the Lagrange multipliers $(z_h,Z_h)$ converge to zero with a rate that is proportional to the space-time mesh parameter $h$. Moreover, we have
$$\mathcal R(u_h) \lesssim h\|(g_0,g_1)\|_{H^{4}(\Omega)\times H^3(\Omega)},$$
implying that the initial and final states of the discrete solution $u_h$ converge to the desired values at the optimal rate. In order to prove this proposition, we need the following lemma.
\begin{lem}
\label{Aupperbound}
Let $x_h\in \mathbb V_h^{4N}$ be as in \eqref{x}. Then:
\[ \mathcal A(x_h;y) \lesssim h \|(g_0,g_1)\|_{H^{4}(\Omega)\times H^3(\Omega)} \tnorm y \tnorm_C \quad \forall \, y \in \mathbb V_h^{4N}.\]
\end{lem}

\begin{proof}
Let $y=(v,V,w,W)$. We can use the Euler-Lagrange equation \eqref{EL} to write
$$\mathcal A(x_h;y)=a_h(v^0,g_0)+(\p_\tau v^1,g_1)_h-\mathcal A((\pi_h u_*,\pi_h U_*,0,0);y)=S_1+S_2+S_3+S_4+S_5,$$
where 
\[
\begin{aligned}
S_1&=-\mathcal G(\pi_h u_*,w)+\tau \sum_{n=2}^N(\chi_\omega U_*^n,w^n)_h,\\           
S_2&=-\mathcal G^*(W,\pi_h U_*),\\
S_3&=\tau \sum_{n=2}^N(\chi_h\pi_h U_*^n-\chi_\omega U_*^n,w^n)_h,\\
S_4&=-a_h(\pi_h u_*^N,v^N)+a_h(g_0-\pi_h u_*^0,v^0)-(\p_\tau \pi_h u_*^N,\p_\tau v^N)_h+(g_1-\p_\tau \pi_h u_*^1,\p_\tau v^1)_h,\\
S_5&=-\mathcal A_1((\pi_h U_*,0,0);V).
\end{aligned}
\]
For the term $S_1$, we first observe that $u_*$ satisfies the equation 
$$ \tau \sum_{n=2}^N \int_\Omega \left(\p^2_t u_*^n\cdot w^n+\nabla u_*^n\cdot \nabla w^n\right)\,dx=\tau \sum_{n=2}^N\int_\Omega \chi_\omega U_*^n\cdot w^n\,dx,$$
where $u_*^n=u_*(n\tau,\cdot)$ and $U_*^n=U_*(n\tau,\cdot)$ and we are identifying $u_*,U_*$ with their extensions $u^e_*, U^e_*$. The test function $w^n$ is extended to $\Omega\cup\Omega_h$ by setting it equal to zero outside $\Omega_h$. Note that since $w \in H^1_0(\Omega_h)$ the extended function $w^n$ belongs to $H^1_0(\Omega_h\cup\Omega)$. Together with the definition of the interpolator $\pi_h$, we can write
\[
S_1=\underbrace{\tau\sum_{n=2}^N((1-\pi_h)\p_\tau^2u_*^n,w^n)_h}_{I_1}+\underbrace{\tau \sum_{n=2}^N(\p^2_t u_*^n-\p^2_\tau u_*^n,w^n)_h}_{I_2}+\tau\sum_{n=2}^N(\varsigma_E^n,w^n)_{\Omega_h\setminus\Omega},
\]
where $\varsigma_E u_*^n=E\,(\chi_\omega U_*^n)-\Box E\,u_*^n$. We utilize \cite[Lemma 3.1]{BFO} to write
$$ |(\varsigma_E^n,w^n)|\leq \|\varsigma_E^n\|_{\Omega_h\setminus\Omega}\|w^n\|_{\Omega_h\setminus\Omega}\lesssim (\|\chi_\omega U_*^n\|_{L^2(\Omega)}+\|\p^2_t u_*^n\|_{L^2(\Omega)}+\|u_*^n\|_{H^2(\Omega)})h^2\|\nabla w^n\|_h.$$
Thus using Theorem~\ref{smoothness} for $u_*$, $U_*$ and the inverse discrete inequality in \eqref{reversepoincare} for $w^n$, we can write
\[
\begin{aligned}
|\tau\sum_{n=2}^N(\varsigma_E^n,w^n)_{\Omega_h\setminus\Omega}|^2&\lesssim h^2(\|U_*\|^2_{\mathcal C(0,T;L^2(\Omega))}+\|u_*\|^2_{\mathcal C^2(0,T;L^2(\Omega))}+\|u_*\|^2_{\mathcal C(0,T;H^2(\Omega))})(\tau\sum_{n=2}^N\|w^n\|^2_h)\\
&\lesssim h^2\|(g_0,g_1)\|^2_{H^4(\Omega)\times H^3(\Omega)}\tnorm(0,0,w,0)\tnorm^2_C.
\end{aligned}
\]
Using the same analysis as that of the terms $I_1,I_2$ in
\cite[Proposition~4.3]{BFO} we obtain
\[
\begin{aligned}
|I_1|^2&\lesssim h^2 (\int_0^T \|\nabla \p^2_t u_*\|_{L^2(\Omega)}^2\,dt)(\tau \sum_{n=2}^N\|w^n\|_h^2),\\
|I_2|^2&\lesssim \tau^2(\int_0^T\|\p^3_\tau u_*\|_{L^2(\Omega)}^2\,dt)(\tau \sum_{n=2}^N\|w^n\|_h^2).
\end{aligned}
\]
We can therefore use Theorem~\ref{smoothness} to conclude that
$$|S_1| \lesssim h\|(g_0,g_1)\|_{H^{4}(\Omega)\times H^3(\Omega)}\tnorm(0,0,w,0)\tnorm_C.$$
For the term $S_2$, we first note that 
$$\tau \sum_{n=0}^{N-2}\int_{\Omega}\left( W^n\cdot \p^2_t U^n_*+\nabla W^n\cdot \nabla U_*^n\right)\,dx=0,$$
where we recall the notation $U_*^{n}(\cdot)=U_*(n\tau,\cdot)$. Therefore, using the definition of the interpolator $\pi_h$ we can write
$$S_2=\tau\sum_{n=0}^{N-2}(W^n,(1-\pi_h)\tilde{\p}_\tau^2U_*^n)_h+\tau \sum_{n=0}^{N-2}(W^n,\p^2_t U_*^n-\tilde{\p}^2_\tau U_*^n)_h+\tau\sum_{n=0}^{N-2}(\tilde{\varsigma}_E^n,W^n),$$
where $\tilde{\varsigma}_E^n=-\Box E\,U_*^n$. Analogous to the term $S_1$ we obtain the bound
$$|S_2|\lesssim h\|(g_0,g_1)\|_{H^{4}(\Omega)\times H^3(\Omega)}\tnorm(0,0,0,W)\tnorm_C,$$
where we have used Theorem~\ref{smoothness} again. For the term $S_3$, we write
\begin{multline*}
|S_3|=|\tau \sum_{n=2}^N((\chi_h-\chi_\omega)\pi_h U_*^n,w^n)_h+\tau \sum_{n=2}^N(\chi_\omega(\pi_h U_*^n- U_*^n),w^n)_h|\\
\lesssim h\|(g_0,g_1)\|_{H^{4}(\Omega)\times H^3(\Omega)}\tnorm(0,0,w,0)\tnorm_C,
\end{multline*}
where we have used the bound \eqref{eq:disc_chi} together with Lemma~\ref{interpolator} for the first term and Lemma~\ref{interpolator} for the second term. For $S_4$, we use the the fact that $u_*(0,\cdot)=g_0$, $\p_t u(0,\cdot)=g_1$, and equation \eqref{ctrl} with the approximation properties of $\pi_h$ to deduce that
$$|S_4| \leq 2 \sqrt{\mathcal R(\pi_h u_*)\mathcal R(v)}\lesssim h\|(g_0,g_1)\|_{H^{4}(\Omega)\times H^3(\Omega)}\tnorm(v,0)\tnorm_R,$$
where we have used the fact that $\pi_h u_*^N=0$ and Theorem~\ref{smoothness} to obtain the following bounds:
\[
\begin{aligned}
\|\p_\tau \pi_h u_*^N\|_h &=\|(\pi_h-1)\p_\tau u_*^N+\p_\tau u_*^N-\pd_tu_*^N\|_h \lesssim h \|(g_0,g_1)\|_{H^{4}(\Omega)\times H^3(\Omega)},\\
\|\p_\tau \pi_h u_*^1-g_1\|_h &=\|(\pi_h-1)\p_\tau u_*^1+\p_\tau u_*^1-g_1\|_h \lesssim h \|(g_0,g_1)\|_{H^{4}(\Omega)\times H^3(\Omega)},\\
\|\nabla (\pi_h u_*^0- g_0)\|_h &=\|\nabla (\pi_h-1) u_*^0\|_h \lesssim h \|(g_0,g_1)\|_{H^{4}(\Omega)\times H^3(\Omega)}.
\end{aligned}
\]
Finally, for the term $S_5$, we write
$$|S_5| \leq 2\sqrt{\mathcal J_1(\pi_h U_*)\mathcal J_1(V)}\lesssim h\|(g_0,g_1)\|_{H^{4}(\Omega)\times H^3(\Omega)}\tnorm(0,V)\tnorm_R,$$
where we have used the following bounds for $\mathcal J_1(\pi_h U_*)$:
\[
\begin{aligned}
\tau \sum_{n=1}^N\|\tau\nabla\p_\tau \pi_h U_*\|_h^2 &\lesssim \tau^2\|U_*\|_{H^1(0,T;H^1(\Omega))}^2\lesssim h^2\|(g_0,g_1)\|_{H^{4}(\Omega)\times H^3(\Omega)}^2,\\
\|h\nabla \p_\tau \pi_h U_*^N\|_h^2 &\lesssim h^2\|U_*\|^2_{H^2(0,T;H^1(\Omega))}\lesssim h^2\|(g_0,g_1)\|_{H^{4}(\Omega)\times H^3(\Omega)}^2,\\
\|h\nabla \p_\tau \pi_h U_*^1\|_h^2 &\lesssim h^2\|U_*\|^2_{H^2(0,T;H^1(\Omega))}\lesssim h^2\|(g_0,g_1)\|_{H^{4}(\Omega)\times H^3(\Omega)}^2,\\
\|h\nabla \pi_h U_*^N\|_h^2 &\lesssim h^2\|U_*\|_{H^1(0,T;H^1(\Omega))}\lesssim h^2\|(g_0,g_1)\|_{H^{4}(\Omega)\times H^3(\Omega)}^2.
\end{aligned}
\]
Combining the estimates yields the claim. 
\end{proof}

\begin{proof}[Proof of Proposition~\ref{discretesol}]
Let $x_h$ be as in equation \eqref{x}. Using Proposition~\ref{estimate}, there exists $y \in \mathbb V_h^{4N}$ such that 
$$ \mathcal A(x_h;y) \gtrsim (\tnorm(\tilde{u}_h,\tilde{U}_h)\tnorm_R +h\tnorm \tilde{u}_h \tnorm_{F}+ h\tnorm \tilde{U}_h \tnorm_{F'}+\tnorm z_h \tnorm_{D}+\tnorm Z_h \tnorm_{D'})\tnorm y \tnorm_C.$$
Combining this estimate with Lemma~\ref{Aupperbound} yields the claim.
\end{proof}
Lemma~\ref{interpolator} can be used together with Theorem~\ref{smoothness} to obtain the following corollary.
\begin{cor}
\label{discretesol2}
$$\tnorm u_h \tnorm_F+\tnorm U_h \tnorm_{F'} \lesssim \|(g_0,g_1)\|_{H^{4}(\Omega)\times H^3(\Omega)}.$$
\end{cor}

We are now ready to state two key ingredients of the proof of Theorem~\ref{t1}. The first estimate is regarding the error function $\tilde{u}_h$. We recall that for each $n$, $\tilde{u}^n_h$ is extended to $\Omega\cup \Omega_h$ by setting it to be zero outside $\Omega_h$.
\begin{lem}
\label{wellposed}
Let $\tilde{u}_h$ be as defined in equation \eqref{x}. For $n=1,\ldots,N$, the following estimate holds:
$$ \|\tilde{u}_h^n\|^2_{H^1(\Omega)}+\|\p_\tau \tilde{u}_h^n\|^2_{L^2(\Omega)} \lesssim h^2\|(g_0,g_1)\|^2_{H^4(\Omega)\times H^3(\Omega)}+\tau\sum_{n=2}^N\|\sqrt{\chi_h}\tilde{U}_h^n\|^2_{h}.$$ 
\end{lem} 
\begin{proof}
Note that Proposition~\ref{discretesol} implies that $\|\tilde{u}^0_h\|_{H^1(\Omega)}+\|\p_\tau \tilde{u}^1_h\|_{L^2(\Omega)}\lesssim h\|(g_0,g_1)\|_{H^{4}(\Omega)\times H^3(\Omega)}$. Recall that $u_h$ satisfies the discrete wave equation
$$ \tau \sum_{n=2}^N \left((\p^2_\tau u_h^n,w^n)_h+a_h(u_h^n,w^n)\right)=\tau \sum_{n=2}^N (\chi_h U_h^n,w^n)_h,$$
and that $u_*$ solves the wave equation \eqref{pf}. Thus standard discrete energy estimates for the wave equation apply to
derive the claimed inequality (see for example
  \cite[Lemma 6]{Dup73}).
\end{proof}

Next, we state the following approximate observability estimate for the error function $\tilde{U}_h$ defined in \eqref{x}. We remind the reader that for each $n$, $U_h^n$ is extended to $\Omega\cup\Omega_h$ by setting it to zero outside $\Omega_h$.
\begin{prop}
\label{discreteobs}
Let $\tilde{U}_h$ be as defined in equation \eqref{x}. For $n=1,\ldots,N$, the following estimate holds:
$$\|\tilde{U}_h^n\|^2_{L^2(\Omega)}+\|\p_\tau \tilde{U}_h^n\|^2_{H^{-1}(\Omega)} \lesssim h^2 \|(g_0,g_1)\|^2_{H^4(\Omega)\times H^3(\Omega)} + \tau \sum_{n=2}^N \|\sqrt{\chi_h}\tilde{U}_h^n\|^2_{h}.$$ 
\end{prop}
\begin{proof}

We begin by defining the continuous piece-wise affine function
$$\hat{U}_h(t,\cdot)=U_h^{n-1}(\cdot)+(t-t_{n-1})\p_\tau U_h^n(\cdot) \quad \text{for}\quad t \in [t_{n-1},t_{n}],$$
for $n=1,\ldots,N$. 
Let $\mathcal{E}=\hat{U}_h-U_*$ and define the bounded linear functional $R$ through
\bel{bounda}
\langle R,W\rangle = \int_0^T \int_{\Omega} (-\partial_t \mathcal E \cdot \partial_t W + \nabla \mathcal E \cdot \nabla W) \,dx \,dt    \hspace{5mm} \forall\, W \in H^1_0(\mathcal{M}). 
\ee
For the remainder of this proof, we will identify $U_*$ with its $E$-extension to $(0,T)\times (\Omega\cup\Omega_h)$. Applying Theorem~\ref{observability} with $\tilde{\mathcal O}=(0,T)\times \omega_\delta$ and using the fact that $\chi_\omega=1$ on $\omega_\delta$, we obtain
\bel{rf2}
\sup_{t\in[0,T]}(\|\mathcal E(t,\cdot)\|^2_{L^2(\Omega)}+\|\p_t \mathcal E(t,\cdot)\|^2_{H^{-1}(\Omega)}) \lesssim \|\sqrt{\chi_{\omega}}\mathcal E\|^2_{L^2(\M)} +
\|R\|^2_{H^{-1}(\M)}+\|\mathcal E\|^2_{L^2((0,T)\times \pd \Omega)}.
\ee
We proceed to prove the following bounds:
\begin{equation}
\label{boundarybound}
\|\mathcal E\|_{L^2((0,T)\times \pd \Omega)} \lesssim h\|(g_0,g_1)\|_{H^{4}(\Omega)\times H^3(\Omega)},
\end{equation}
\begin{equation}
\label{global}
|\langle R,W\rangle| \lesssim h\|(g_0,g_1)\|_{*}\|W\|_{H^1(\mathcal{M})} \quad \forall W \in H^1_0(\M),
\end{equation}
\bel{interiorbounda}
\|\sqrt{\chi_{\omega}}\mathcal E\|^2_{L^2(\M)}\lesssim h^2\|(g_0,g_1)\|^2_{H^4(\Omega)\times H^3(\Omega)}+\tau \sum_{n=2}^N \|\sqrt{\chi_h}\tilde{U}_h^n\|^2_{h}.
\ee
The proposition then follows by writing
$\tilde{U}_h^n(\cdot)=\mathcal
E(n\tau,\cdot)+\mathcal{E}_*(n\tau,\cdot)$ with $\mathcal{E}_*(n\tau,\cdot)=U_*(n\tau,\cdot)-\pi_h U_*^n(\cdot)$. The term $\mathcal{E}_*(n\tau,\cdot)
  $ may then be bounded  by $h\|(g_0,g_1)\|_{H^{4}(\Omega)\times H^3(\Omega)}$ using
Lemma~\ref{interpolator} and Taylor development in time followed by
Theorem~\ref{smoothness}.

First we prove the estimate \eqref{boundarybound}. Recalling that $U_*|_{(0,T)\times \pd \Omega}=0$, we write
$$ \|\mathcal E\|^2_{L^2((0,T)\times \pd \Omega)}=\|\hat{U}_h\|^2_{L^2((0,T)\times\pd \Omega)}\lesssim \tau \sum_{n=0}^N\|U_h^n\|_{L^2((0,T)\times\pd \Omega)}^2.$$
Applying \cite[Lemma 4.1]{BFO} and Corollary~\ref{discretesol2}, we deduce that
$$ \tau \sum_{n=0}^N\|U_h^n\|_{L^2((0,T)\times\pd \Omega)}^2\lesssim\tau\sum_{n=0}^Nh^2\|\nabla U_h^n\|_h^2\lesssim h^2\tnorm U_h \tnorm_{F'}^2\lesssim h^2\|(g_0,g_1)\|_{H^{4}(\Omega)\times H^3(\Omega)}^2.$$

Let us now consider the estimate \eqref{global}. We introduce the notation $\bar{W}^n$ through
\bel{ave}
\begin{aligned}
&\,\bar{W}^n(\cdot)=\frac{1}{\tau}\int_{t_{n-1}}^{t_{n}}W(t,\cdot)\,dt \quad \text{for} \quad n=1,\ldots,N.\\
\end{aligned}
\ee
Note that using the Poincar\'{e} and Cauchy-Schwarz inequalities, we have:
\bel{poincare}
\begin{aligned}
\tau \sum_{n=1}^{N-1}\|W^n-\bar{W}^n\|_{L^2(\Omega)}^2 &\lesssim \tau^2 \|W\|^2_{H^1(0,T;L^2(\Omega))},\\
\|\nabla \bar{W}^n\|^2_{L^2(\Omega)}&\lesssim \tau^{-1}\|W\|^2_{L^2(0,T;H^1(\Omega))}\quad \text{for}\quad n=1,\ldots,N,\\
\|\bar{W}^N\|^2_{L^2(\Omega)} &\lesssim \tau \|W\|^2_{H^1(0,T;L^2(\Omega))}.
\end{aligned}
\ee
Since $\Box U_*=0$ on $\M$ and $U_h=0$ on $(0,T)\times (\Omega\setminus\Omega_h)$, we have:
\[
\begin{aligned}
\langle R,W\rangle=&\,\int_0^T (-\p_t \hat{U}_h,\p_t W)_h+a_h(\hat{U}_h,W)\\
=&\,\sum_{n=1}^N \int_{t_{n-1}}^{t_n}\left[-(\p_\tau U_h^n,\p_t W)_h+a_h(U_h^{n-1},W)+(t-t_{n-1})a_h(\p_\tau U_h^n,W)\right] \,dt\\
=&\,\tau \sum_{n=1}^{N-1} (\p^2_\tau U_h^{n+1},W^n)_h+\tau \sum_{n=1}^N a_h(U_h^{n-1},\bar{W}^n)+\sum_{n=1}^N\int_{t_{n-1}}^{t_n}(t-t_{n-1})a_h(\p_\tau U_h^n,W) \,dt\\
=&\,\underbrace{\tau\sum_{n=1}^{N-1}(\p^2_\tau U_h^{n+1},W^n-\bar{W}^n)_h}_{I}+\underbrace{\tau a_h(U_h^{N-1},\bar{W}^N)}_{II}\\
&+\underbrace{\tau \sum_{n=0}^{N-2}\left((\tilde{\p}^2_\tau U_h^n,\bar{W}^{n+1})_h+a_h(U_h^n,\bar{W}^{n+1})\right)}_{III}+\underbrace{\sum_{n=1}^N\int_{t_{n-1}}^{t_n}(t-t_{n-1})a_h(\p_\tau U_h^n,W) \,dt}_{IV}.
\end{aligned}
\]
We will first proceed to bound each of the terms I, III and IV and then treat the term $II$. For the term $I$, we use the Cauchy-Schwarz inequality to write
\[ |I|^2 \leq (\tau \sum_{n=2}^{N}\|\p^2_\tau U_h^n\|_h^2)(\tau \sum_{n=1}^{N-1}\|W^n-\bar{W}^n\|_{L^2(\Omega)}^2)\lesssim \tau^2\|(g_0,g_1)\|^2_{H^4(\Omega)\times H^3(\Omega)}\|W\|^2_{H^1(\M)},\]
where we have used the first bound in \eqref{poincare} followed by Corollary~\ref{discretesol2}.
For the term $III$, we first note that $U_h$ satisfies the equation
$$\tau \sum_{n=0}^{N-2}\left((\tilde{\p}^2_\tau U_h^n,\pi_h \bar{W}^{n+1})_h+a_h(U_h^n,\pi_h \bar{W}^{n+1})\right)=0,$$
which together with the definition of the interpolator $\pi_h$ and Corollary~\ref{discretesol2}, implies that
\[
\begin{aligned}
|III|^2=|\tau\sum_{n=0}^{N-2}(\tilde{\p}^2_\tau U_h^n,(\pi_h-1)\bar{W}^{n+1})_h|^2\leq\,&  (\tau \sum_{n=2}^{N}\|\p^2_\tau U_h^n\|_h^2)(\tau\sum_{n=0}^{N-2}\|(\pi_h-1)\bar{W}^{n+1}\|^2)\\
&\lesssim\tau^2\|(g_0,g_1)\|^2_{H^4(\Omega)\times H^3(\Omega)}\|W\|^2_{H^1(\M)}.
\end{aligned}\]
For the term $IV$, we use the Cauchy-Schwarz inequality to write
$$|IV|^2\lesssim \tau^2\left(\tau \sum_{n=1}^N\|\p_\tau\nabla U_h^n\|_h^2 \right)\|W\|^2_{H^1(\M)} \lesssim \tau^2\|(g_0,g_1)\|^2_{H^4(\Omega)\times H^3(\Omega)}\|W\|^2_{H^1(\M)}.$$
where we have used Corollary~\ref{discretesol2} again. It remains to bound the term $II$. Here, we use the fact that $U_h$ is solving the Euler-Lagrange equation \eqref{EL} again to write
$$\tau a_h(U_h^{N-1},\bar{W}^N)=\tau a_h(U_h^{N-2},\bar{W}^N)+\tau^2 a_h(\p_\tau U_h^{N-1},\bar{W}^N).$$
But using Corollary~\ref{discretesol2} we have the bound $\|\nabla \p_\tau U^{N-1}\|_h \lesssim \tau^{-\frac{1}{2}}\|(g_0,g_1)\|_{H^{4}(\Omega)\times H^3(\Omega)}$, which together with the bound \eqref{poincare}
implies that 
$$\tau a_h(\tau\p_\tau U_h^{N-1},\bar{W}^N)\lesssim \tau \|(g_0,g_1)\|_{H^{4}(\Omega)\times H^3(\Omega)}\|W\|_{H^1(\M)}.$$
For the remaining term, we observe that
$$ \tau a_h(U_h^{N-2},\bar{W}^N)=\tau a_h(U_h^{N-2},\pi_h\bar{W}^N)=\underbrace{-\tau(\tilde{\p}^2_\tau U_h^{N-2},(\pi_h-1)\bar{W}^N)_h}_{S_1}\underbrace{-\tau(\tilde{\p}^2_\tau U_h^{N-2},\bar{W}^N)_h}_{S_2}.$$ 
To bound the term $S_1$, we use Lemma~\ref{interpolator} and the second bound in \eqref{poincare} to obtain
\[
|S_1|\lesssim \tau \|\p^2_\tau U_h^{N}\|_h\tau^{\frac{1}{2}}\|W\|_{H^1(\M)}\lesssim \tau\|(g_0,g_1)\|_{H^{4}(\Omega)\times H^3(\Omega)}\|W\|_{H^1(\M)},
\]
where we have used Corollary~\ref{discretesol2} to write the bound $\|\p^2_\tau U_h^{N}\|_h\lesssim\tau^{-\frac{1}{2}}\|(g_0,g_1)\|_{H^{4}(\Omega)\times H^3(\Omega)}$. Finally for the term $S_2$, we write
\[
|S_2|\lesssim \tau \|\p^2_\tau U_h^{N}\|_h\|\bar{W}^N\|_h\lesssim \tau\|(g_0,g_1)\|_{H^{4}(\Omega)\times H^3(\Omega)}\|W\|_{H^1(\M)},
\]
where we have used the third bound in \eqref{poincare} together with Corollary~\ref{discretesol2} in the last step. This completes the proof of bound \eqref{global}. 

To prove \eqref{interiorbounda}, we first define the piece-wise constant time interpolant $\pi_0$ as follows:
\[ \pi_0 v^n=v(t_n) \quad \text{for} \quad t \in (t_{n-1},t_n]\quad \text{and} \quad n=1,\ldots,N.\]
This interpolant satisfies the bound
\[\|\pi_0v-v\|_{L^2(0,T)}\lesssim \tau \|v\|_{H^1(0,T)}.\]  
Note that by adding and subtracting $\chi_h$ and using
\eqref{eq:disc_chi} we have
\begin{multline*}
\|\sqrt{\chi_\omega}\mathcal E\|^2_{L^2(\M)} \leq \|\chi_\omega -
\chi_h\|^2_{L^\infty(\Omega)} \|\mathcal E\|^2_{L^2(\M)} +
\|\sqrt{\chi_h}\mathcal E\|^2_{L^2(\M)} \\
\lesssim h^2 \int_0^T (\|U_*\|_{L^2(\Omega)}^2 +
\|\hat U_h\|_{L^2(\Omega)}^2)\,dt+ \|\sqrt{\chi_h}\mathcal E\|^2_{L^2(\M)}.
\end{multline*}
Observe that
\[
 h^2 \int_0^T
\|\hat U_h\|_{L^2(\Omega)}^2\,dt \lesssim h^2 \tnorm \tilde U_h \tnorm_{F'}^2.
\]
Using Corollary~\ref{discretesol2} and Theorem~\ref{smoothness}  it
follows that
\[ \|\sqrt{\chi_\omega}\mathcal E\|^2_{L^2(\M)} \lesssim h^2\|(g_0,g_1)\|^2_{H^4(\Omega)\times H^3(\Omega)}+\|\sqrt{\chi_h}\mathcal E\|^2_{L^2(\M)}.\]
Furthermore,
$$\|\sqrt{\chi_h}\mathcal E\|^2_{L^2(\M)}\leq\, C(h^2+\tau^2)\|U_*\|_{H^1(\M)}^2+\int_0^T\|\sqrt{\chi_h}\pi_0\pi_hU_*-\sqrt{\chi_h}\hat{U}_h\|_h^2\,dt.$$
To bound the second term we write
\[
\begin{aligned}
\int_0^T\|\sqrt{\chi_h}\pi_0\pi_hU_*-\sqrt{\chi_h}\hat{U}_h\|_h^2\,dt&\lesssim \int_0^T\|\sqrt{\chi_h}\pi_0\pi_hU_*-\pi_0\sqrt{\chi_h}\hat{U}_h\|_h^2\,dt+\int_0^T\|\pi_0\hat{U}_h-\hat{U}_h\|_h^2\,dt\\
&=\tau \sum_{n=1}^N\|\sqrt{\chi_h}\tilde{U}_h^n\|_h^2+\sum_{n=1}^N\int_{t_{n-1}}^{t_n}\|\pi_0\hat{U}_h-\hat{U}_h\|_h^2\,dt.
\end{aligned}
\]
It suffices to bound the second term of the right hand side. Using the piece-wise linearity of $\hat{U}_h$ we observe that
\begin{multline*}
\sum_{n=1}^N\int_{t_{n-1}}^{t_n}\|\pi_0\hat{U}_h-\hat{U}_h\|_h^2\,dt=\sum_{n=1}^N\int_{t_{n-1}}^{t_n}\|(t-t_n)\p_\tau U_h^n\|_h^2\,dt \leq \tau \sum_{n=1}^N\|\tau\p_\tau U_h^n\|_h^2\\
\lesssim h^2\|(g_0,g_1)\|_{H^{4}(\Omega)\times H^3(\Omega)}^2,
\end{multline*}
where we have used Corollary~\ref{discretesol2} in the last step. This completes the proof of the bound \eqref{interiorbounda}.
\end{proof}

\section{A strong a priori error estimate and proof of the main theorem}
\label{strong_est_section}
This section is concerned with the proof of the main theorem. The idea is to use the approximate discrete observability estimate together with an improved coercivity estimate to produce a stronger error estimate, stated as follows:
\begin{prop}
\label{coercivityf}
Let $\tilde{U}_h$ be as defined in \eqref{x}. The following residual estimate holds:
\[
\tau \sum_{n=2}^N\|\sqrt{\chi_h}\tilde{U}_h^n\|^2_h \lesssim h^2 \|(g_0,g_1)\|^2_{H^4(\Omega)\times H^3(\Omega)}.
\]
\end{prop}
The proof of the main theorem follows from Proposition~\ref{coercivityf}. Indeed, note that the first claimed inequality in Theorem~\ref{t1} follows from combining Propositions~\ref{discreteobs} and \ref{coercivityf}, while the second claimed inequality follows from combining Lemma~\ref{wellposed} and Proposition~\ref{coercivityf}. We proceed to prove Proposition~\ref{coercivityf}. This will be divided into parts. We define the refined test function 
\bel{testfcns1}
\hat{y}=(\hat{v},\hat{V},\hat{w},\hat{W})
\ee 
through the expressions
\[
\begin{aligned}
\hat{v} = \tilde{u}_h+\gamma v,   \quad \quad &\hat{V}=\tilde{U}_h+\alpha V,\\
\hat{w}= -z_h-\gamma \tilde{U}_h+\alpha h^2 w, \quad \quad     &\hat{W}=-Z_h+\gamma h^2 W,
\end{aligned}
\]
where $\gamma>\alpha>0$ and $v,V,w,W$ are chosen as in Lemma~\ref{energyuUzZ} in terms of the discrete functions $z_h$, $Z_h$, $\tilde{u}_h$ and $\tilde{U}_h$ respectively. Let us also define a norm on $\mathbb V_h^{4N}$ through the expression
\bel{normfinal}
\begin{aligned}
\tnorm(u,U,z,Z)\tnorm^2_{S}=\, & \tau \sum_{n=2}^N\|\sqrt{\chi_h}U^n\|^2_h+\tnorm(u,U)\tnorm^2_R +h^2\tnorm u \tnorm^2_{F} \\
&+ h^2\tnorm U \tnorm^2_{F'}+\tnorm z \tnorm^2_{D}+\tnorm Z \tnorm^2_{D'}
\end{aligned}
\ee

We have the following three lemmas. These will be subsequently used to prove Proposition~\ref{coercivityf}.
\begin{lem}
\label{upperboundb}
Let $x_h$, $\hat{y}$ be defined as in equations \eqref{x} and \eqref{testfcns1} respectively. The following estimate holds:
$$\tnorm \hat{y} \tnorm_{C}\lesssim h\|(g_0,g_1)\|_{H^{4}(\Omega)\times H^3(\Omega)}+\tnorm x_h \tnorm_{S}.$$
\end{lem}

\begin{proof}
Recall that $$\tnorm \hat{y} \tnorm^2_C=\tnorm(\hat{v},0)\tnorm_R^2+\tnorm(0,\hat{V})\tnorm_R^2+\tau\sum_{n=2}^N\|\hat{w}^n\|_h^2+\tau \sum_{n=0}^{N-2}\|\hat{W}^n\|_h^2.$$
We proceed to bound each of the four terms appearing on the right hand side. Indeed, using equation \eqref{testfcns1} and the Cauchy-Schwarz inequality we have the bounds 
\[
\begin{aligned}
\tnorm(\hat{v},0)\tnorm^2_R &\lesssim \tnorm(\tilde{u}_h,0)\tnorm^2_R + \|\nabla v^N\|_h^2+\|\p_\tau v^N\|_h^2+\|\nabla v^1\|_h^2+\|\p_\tau v^1\|_h^2 \\
&\lesssim   \tnorm(\tilde{u}_h,0)\tnorm^2_R+\|\nabla v^N\|_h^2+\|\p_\tau v^N\|_h^2 \lesssim  \tnorm(\tilde{u}_h,0)\tnorm^2_R+\tnorm z_h \tnorm_D^2,\\
\tnorm(0,\hat{V})\tnorm_R^2 &\lesssim \tnorm(0,\tilde{U}_h)\tnorm_R^2+ h^2 \tau\sum_{n=1}^N \|\p_\tau \nabla V^n\|_h^2+h^2\|\p_\tau \nabla V^1\|_h^2\\
&\lesssim \tnorm(0,\tilde{U}_h)\tnorm_R^2+\tau \sum_{n=0}^{N-2}\|Z_h\|_h^2+\|Z_h^0\|_h^2\lesssim  \tnorm(0,\tilde{U}_h)\tnorm_R^2+\tnorm Z_h \tnorm_{D'}^2,\\
\tau\sum_{n=0}^{N-2}\|\hat{W}^n\|_h^2 &\lesssim \tau\sum_{n=0}^{N-2}\|Z_h^n\|_h^2+h^2\tnorm \tilde{U}_h \tnorm_{F'}^2,\\
\tau\sum_{n=2}^N\|\hat{w}^n\|_h^2\,&\lesssim\tau \sum_{n=2}^N\|z_h^n\|_h^2+h^2\tnorm \tilde{u}_h \tnorm_{F}^2+\tau \sum_{n=2}^N\|\tilde{U}_h\|_h^2,
\end{aligned}
\]
Applying Proposition~\ref{discreteobs} in the last bound and using the definition of the $\tnorm\cdot\tnorm_{S}$ norm yields the claim.
\end{proof}
\begin{lem}
\label{mixedterm}
The following estimate holds:
\[
|\mathcal G(\tilde{u}_h,\tilde{U}_h)| \lesssim h\|(g_0,g_1)\|_{H^{4}(\Omega)\times H^3(\Omega)}\left(h\|(g_0,g_1)\|_{H^{4}(\Omega)\times H^3(\Omega)}+\tnorm x_h \tnorm_{S}\right).
\]
\end{lem}
\begin{proof}
We begin by using the discrete version of Leibniz rule to write
\[
\begin{aligned}
\mathcal G(\tilde{u}_h,\tilde{U}_h)&=\mathcal G^*(\tilde{u}_h,\tilde{U}_h)\underbrace{-(\tilde{u}_h^{N-1},\p_\tau \tilde{U}_h^N)_h+(\p_\tau\tilde{u}_h^N,\tilde{U}_h^N)_h+(\tilde{u}_h^0,\p_\tau \tilde{U}_h^1)_h-(\p_\tau \tilde{u}_h^1,\tilde{U}_h^1)_h}_{I}\\
&\underbrace{-\tau a_h(\tilde{u}_h^0,\tilde{U}_h^0)-\tau a_h(\tilde{u}_h^1,\tilde{U}_h^1)+\tau a_h(\tilde{u}_h^{N-1},\tilde{U}_h^{N-1})+\tau a_h(\tilde{u}_h^N,\tilde{U}_h^N)}_{II}.
\end{aligned}
\]
Now using the fact that $\mathcal G^*(W,U_h)=0$ for all $W$ (see \eqref{EL}), we obtain
$$\mathcal G^*(\tilde{u}_h,\tilde{U}_h)=-\mathcal G^*(\tilde{u}_h,\pi_h U_*),$$
which is identical to the term $S_2$ in Lemma~\ref{Aupperbound}. Therefore we have the bound
\bel{bulk}
\begin{aligned}
|\mathcal G^*(\tilde{u}_h,\pi_h U_*)|^2 &\lesssim h^2\|(g_0,g_1)\|^2_{H^4(\Omega)\times H^3(\Omega)}(\tau \sum_{n=0}^N \|\tilde{u}_h^n\|^2_h)\\
&\lesssim h^2\|(g_0,g_1)\|^2_{H^4(\Omega)\times H^3(\Omega)}(h^2\|(g_0,g_1)\|^2_{H^4(\Omega)\times H^3(\Omega)}+\tau\sum_{n=2}^N\|\sqrt{\chi_h}\tilde{U}_h^n\|^2_h),
\end{aligned}
\ee
where we have used Lemma~\ref{wellposed} in the last step. To analyze the terms $I$ and $II$, we first observe that:
$$\|\nabla \tilde{u}_h^{1}\|_h=\|\nabla \tilde{u}_h^0 +\tau \nabla\p_\tau\tilde{u}_h^1\|\leq \|\nabla \tilde{u}_h^0\|_h+\kappa \|\p_\tau \tilde{u}_h^1\|_h.$$
$$\|\nabla \tilde{u}_h^{N-1}\|_h=\|\nabla \tilde{u}_h^N - \tau \nabla\p_\tau\tilde{u}_h^N\|\leq \|\nabla \tilde{u}_h^N\|_h+\kappa \|\p_\tau \tilde{u}_h^N\|_h.$$
Now, for the term $I$, we can use Proposition~\ref{discretesol} with Proposition~\ref{discreteobs} to deduce that
\bel{btmix}
|I|^2\lesssim h^2\|(g_0,g_1)\|^2_{H^4(\Omega)\times H^3(\Omega)}(h^2\|(g_0,g_1)\|^2_{H^4(\Omega)\times H^3(\Omega)}+\tau\sum_{n=2}^N\|\sqrt{\chi_h}\tilde{U}_h^n\|^2_h).
\ee
For the term $II$, we use the bounds
$$ \tau \|\nabla \tilde{U}_h^k\|_h \leq \kappa \|\tilde{U}_h^k\|_h\quad \text{for}\quad k=0,1,\ldots,N,$$
and write
$$|II|^2 \lesssim
\tnorm(\tilde{u}_h,0)\tnorm^2_R(\|\tilde{U}_h^0\|^2_h+\|\tilde{U}_h^1\|^2_h+\|\tilde{U}_h^{N-1}\|^2_h+\|\tilde{U}_h^N\|^2_h).$$
We can apply Proposition~\ref{discretesol} together with Proposition~\ref{discreteobs} again to obtain
$$|II|^2\lesssim h^2\|(g_0,g_1)\|^2_{H^{4}(\Omega)\times H^3(\Omega)}(h^2\|(g_0,g_1)\|^2_{H^4(\Omega)\times H^3(\Omega)}+\tau\sum_{n=2}^N\|\sqrt{\chi_h}\tilde{U}_h^n\|^2_h).$$
Combining this with inequalities \eqref{bulk} and \eqref{btmix} completes the proof.
\end{proof}
\begin{lem}
\label{positivityf}
Let $x_h,\hat{y}$ be defined as in equations \eqref{x}, \eqref{testfcns1} respectively. The following estimate holds:
\[
\mathcal A(x_h;\hat{y}) \geq\, C\tnorm x_h \tnorm_{S}^2-C'h\|(g_0,g_1)\|_{H^{4}(\Omega)\times H^3(\Omega)}\left(h\|(g_0,g_1)\|_{H^{4}(\Omega)\times H^3(\Omega)}+\tnorm x_h \tnorm_{S}\right),
\]
where $C,C'>0$ are constants independent of the parameter $h$.
\end{lem}

\begin{proof}
This proof mirrors the proof of Proposition~\ref{estimate}. We start by writing
\bel{er1}
\mathcal A(x_h;\hat{y})=\tnorm(\tilde{u}_h,\tilde{U}_h)\tnorm_R^2+\tau\gamma \sum_{n=2}^N\|\sqrt{\chi_h}\tilde{U}_h\|_h^2-\gamma \mathcal G(\tilde{u}_h,\tilde{U}_h)+\mathcal A(x_h;\tilde{y}),
\ee
where $\tilde{y}=(\gamma v,\alpha V,\alpha h^2w,\gamma h^2W)$ with $v$, $V$, $w$ and $W$ defined as in \eqref{testfcns1}. The analysis of the last term on the right is exactly as in the proof of Proposition~\ref{estimate} and therefore using the same bounds as in that proof, we deduce that
\bel{er2}
\mathcal A(x_h;\tilde{y}) \geq -C_1\gamma\tnorm(\tilde{u}_h,\tilde{U}_h)\tnorm_R^2+C_2\alpha(h^2\tnorm \tilde{u}_h \tnorm^2_{F}+ h^2\tnorm \tilde{U}_h \tnorm^2_{F'}+\tnorm z_h \tnorm^2_{D}+\tnorm Z_h \tnorm^2_{D'})
\ee
for some $C_1,C_2>0$. Finally, combining equations \eqref{er1}--\eqref{er2} and applying Lemma~\ref{mixedterm} yields the claim.
\end{proof}

\begin{proof}[Proof of Proposition~\ref{coercivityf}]
We choose $\hat{y}$ as in equation \eqref{testfcns1}. Lemma~\ref{positivityf} applies and we have
$$\mathcal A(x_h;\hat{y}) \geq C \tnorm x_h \tnorm_{S}^2-C'h\|(g_0,g_1)\|_{H^{4}(\Omega)\times H^3(\Omega)}(h\|(g_0,g_1)\|_{H^{4}(\Omega)\times H^3(\Omega)}+\tnorm x_h \tnorm_{S}).$$
On the other hand, Lemma~\ref{Aupperbound} applies and together with Lemma~\ref{upperboundb}, we write
\begin{multline*}
\mathcal A(x_h,\hat{y}) \lesssim h \|(g_0,g_1)\|_{H^{4}(\Omega)\times H^3(\Omega)}\tnorm \hat{y} \tnorm_C \\
\lesssim h\|(g_0,g_1)\|_{H^{4}(\Omega)\times H^3(\Omega)}\left(\tnorm x_h \tnorm_{S}+h\|(g_0,g_1)\|_{H^{4}(\Omega)\times H^3(\Omega)}\right).
\end{multline*}
Combining these bounds, we note that the following inequality holds:
$$\tnorm x_h \tnorm_{S}^2\lesssim h\|(g_0,g_1)\|_{H^{4}(\Omega)\times H^3(\Omega)}\left(h\|(g_0,g_1)\|_{H^{4}(\Omega)\times H^3(\Omega)}+\tnorm x_h \tnorm_{S}\right).$$
This implies that 
$$\tnorm x_h \tnorm_{S} \lesssim h \|(g_0,g_1)\|_{H^{4}(\Omega)\times H^3(\Omega)}.$$
\end{proof}

\section{Further remarks}
\label{conc_remarks}

\subsection{A comparison with the data assimilation problem}

We start this section with a comparison with our earlier work for the dual problem to null controllability for the wave equation, that is the data assimilation (DA) problem. 

Let us briefly recall the (DA) problem as follows. Let $\omega \subset \Omega$ and consider a solution $u$ to the wave equation \eqref{pf} without the a priori knowledge of the initial data $(g_0,g_1)$. The (DA) problem reads as follows: determine the solution $u$, given the additional piece of data $q= u|_{(0,T)\times \omega}$. 

To solve (DA), one can study the critical points for the Lagrangian 
$$\mathcal L(u,z)= \frac{1}{2}\|u-q\|^2_{L^2(\cO)} + \int_{0}^T\int_\Omega (\pd^2_t u\cdot z +\nabla u\cdot \nabla z)\,dx\,dt$$
where the wave equation is imposed on $u$ through the Lagrange multiplier $z$. Similar to the theory of controllability, existence of a unique minimizer for this functional is guaranteed by an observability estimate on the set $\mathcal O=(0,T)\times \omega$. 

In \cite{BFO}, we considered an approach based on finite element method for numerically solving the (DA) problem, using first order finite elements in space and finite differences in time. Analogously to the current work, this method was based on defining a discrete analogue for the Lagrangian $\mathcal L$ that additionally incorporates numerical stabilizers in $u$. Optimal convergence rates were proven under the assumption that the geometric control condition is satisfied on the set $\mathcal O$ (see \cite[Theorem 4.6]{BFO}). 

Although our approach in solving the null controllability problem here draws similarities to that in \cite{BFO}, in the sense that similar
numerical stabilization terms are used, we outline three of the key differences that makes the control problem more challenging. 

Firstly, for the (DA) problem, the data fitting term $\|u-q\|_{L^2(\cO)}$, incorporated in the Lagrangian functional on the discrete level, gives optimal error bounds on the set $\mathcal O$. Then the error bound on $\mathcal O$, combined with the stability properties for the wave equation (obtained through discrete energy estimates) and the continuum observability estimate (Theorem~\ref{observability}) gives optimal error bounds for the (DA) problem. Notice however that for the null controllability problem, proving error bounds in the set $\mathcal O$ is much harder as the control function $U$ is a priori unknown. In fact error bounds for the control function $U$ in the set $\mathcal O$ are obtained in the final section of the paper (see Proposition~\ref{coercivityf}). In the case of the null controllability problem, it is not a priori clear how the observability estimate appears in the convergence rates. To retrace the steps of the proof, we recall that we first derived a weak a priori error estimate (Proposition~\ref{discretesol}). We then used this estimate combined with the observability estimate to obtain an approximate version of the observability estimate at the discrete level (Proposition~\ref{discreteobs}). Finally, using a key {\em hidden} coercivity estimate in the Lagrangian $\mathcal J$ (see Lemma~\ref{positivityf}) combined with Proposition~\ref{discreteobs}, we were able to obtain the optimal error estimates.

Second key difference is the fact that both the state variable $u$ and
the control function $U$ are unknowns in the controllability problem
whereas in the data assimilation problem the only unknown is the state
variable. A closer inspection
of our work in \cite{BFO}, together with the previous literature on the
Lagrangian formulations of the control problem, suggests that a
discrete version of the more commonly studied continuum Lagrangian
$$ \mathcal L(u,U)=\frac{1}{2}\int_{0}^T\int_\Omega \chi_\omega |U|^2\,dt\,dx  -(g_1,U(0,\cdot))_{L^2(\Omega)}+\langle g_0,\p_t U(0,\cdot) \rangle_{H^1_0(\Omega)\times H^{-1}(\Omega)}$$
$$-\int_{0}^T\int_\Omega\pd^2_t u\cdot U\,dt\,dx -\int_0^T\int_\Omega \nabla u \cdot\nabla U\,dt\,dx$$ 
may yield a numerical method, as long as the correct stabilization terms are incorporated. However, we were unable to derive optimal error estimates for this formulation, mainly due to the inconvenient feature that the critical point for the Lagrange multiplier $U$ in this formulation represents the control function and will not be zero. The Lagrangian formulation employed in this paper introduces two Lagrange multipliers $(z,Z)$, which makes the analysis complete.

Finally, let us emphasize that the data assimilation problem has nice
features on the continuum level that the controllability problem is
lacking. The former problem has a unique solution whereas for the
latter the solution is unique only under additional constraints such
as equation \eqref{dual}. The question of existence and smoothness for the
latter problem are quite trivial. Indeed, existence is guaranteed as
long as the data $q$ comes from an actual
solution to the wave equation, while smoothness follows by requiring that
the data $q$ comes from a smooth solution. However, for the control
problem, existence is a consequence of the deep result by Bardos,
Lebeau and Rauch \cite{BLRII}
while smoothness also requires additional assumptions (see Theorem~\ref{smoothness})

\subsection{Concluding remarks}

We have designed a fully discrete finite element method for the
numerical approximation of the interior null controllability problem
subject to the wave equation. The first order case was considered,
using piecewise affine finite element approximation in space and a
first order finite difference formula in
time. A Tikhonov type regularization was applied to the control
function at initial and final times, but the regularization parameter
was chosen to scale with $h$ in such a way that the perturbation due
to regularization vanishes at a suitable rate. This allowed us to prove
error estimates that are optimal compared to interpolation error, for the state variable and suboptimal
with one order in $h$ for the control variable. Observe however that
the convergence rate of the latter is determined by the norms in the
left hand side of the observability
estimate of Theorem \ref{observability} and the convergence rate of the residual
quantities of the scheme evaluated in the norms of the right hand side
\eqref{boundarybound}--\eqref{interiorbounda}. The former can not be improved.
Since also the bound \eqref{global} is optimal for piecewise affine
approximation it appears that the error in the control
variable is optimal if the continuum stability properties and the
numerical approximation properties are both taken into account.

Let us also remark that no regularization was
applied to the Lagrange multiplier variables $z$ and $Z$, leading to a
system where $(u,U)$ and $(z,Z)$ are only weakly coupled allowing for
solution algorithms using the classical forward-backward solving
approach. Finally, it bears pointing out, that the approach using weakly consistent
regularization, discrete inf-sup stability and observability estimate is
not limited to the first order case,
but can be extended to high order methods, using the modus
operandi designed herein. This requires the introduction of suitable residual
based regularization terms that are weakly consistent to the right
order, which appears to be most feasible in the space-time framework. This is a topic for future work.

\appendix{}
\section{}
\label{const_cut}
In this section, we proceed to construct an example of a smooth non-negative cut-off function that satisfies properties (i)--(iii) in \eqref{cutoff_conditions}.

Let us first recall the boundary normal coordinates near $\p \Omega$ in $\Omega$, that are given by the locally smooth diffeomorphism $F:\R \times \p \Omega\to \Omega$ defined through
$$ F(x',x_n)=\gamma_{x'}(x_n).$$
Here, $x'$ is a point on $\p \Omega$ given in local coordinates by $(x_1,\ldots,x_{n-1})$, and $\gamma_{x'}(\cdot)$ denotes the normal line to the boundary $\p\Omega$ with $\gamma_{x'}(0)=x'$, parametrized in terms of its arc-length. This map gives a local coordinate system such that the points $(x_1,...,x_{n-1},0)$ are on the boundary. 

Let $\omega_\delta$ be defined as in Hypothesis~\ref{hypo} and define $\Gamma= \p\omega \cap \p\Omega$ and $\Gamma_\delta= \p \omega_\delta \cap \p \Omega$. Let $\Gamma'_\delta$ be a small open neighborhood of $\Gamma_\delta$ and let $\Gamma''_\delta$ be a small open neighborhood of $\Gamma'_\delta$ such that $\overline{\Gamma''_\delta}\subset\Gamma$. Let $\epsilon>0$ and choose a smooth $\psi(x')$ on $\p \Omega$ such that $\psi=1$ on $\Gamma'_\delta$ and such that $\supp \psi \subset \Gamma''_\delta$. Subsequently, define a smooth function $\Psi$ in the boundary normal coordinates by $\Psi(x',x_n)=\psi(x')$. Next, choose $\eta\in \mathcal C^{\infty}(\overline{\Omega};[0,1])$ such that
\bel{eta}
\eta(x) = \begin{cases}
        1  & \text{if}\, \dis{(x,\p \Omega)}<\frac{\epsilon}{2},\\
             0  & \text{if} \,\,\dis{(x,\p\Omega)}>\epsilon.
     \end{cases}
\ee
Finally, let $\Phi \in \mathcal C^{\infty}(\overline{\Omega};[0,1])$ be such that $\Phi=1$ on $\omega_\delta$ and $\Phi=0$ on $\Omega\setminus \omega$. 
We now define 
$$ \chi_\omega = \eta\,\Psi + (1-\eta)\,\Phi.$$
Note that the above function is a globally well defined smooth non-negative function on $\Omega$. Indeed, the first term in this expression is supported in an $\epsilon$ neighborhood of the boundary $\p\Omega$ where the boundary normal coordinates are well-defined (if $\epsilon>0$ is sufficiently small). 
Let us now prove that (i)--(iii) hold. Given that $\overline{\Gamma_{\delta}''}$ is contained in $\Gamma$, it follows that for $\epsilon$ small, the first term in the definition of $\chi_\omega$ vanishes on $\Omega \setminus \omega$. The second term also vanishes on $\Omega \setminus \omega$, since $\Phi$ vanishes there. Hence, (i) is satisfied.
To show (ii), we note that for $\epsilon$ sufficiently small, $\eta\Psi=\eta$ on the set $\omega_\delta$ and that $\Phi=1$ on $\omega_\delta$. It follows that $\chi_\omega = \eta+(1-\eta)=1$ on $\omega_\delta$. Finally, to show (iii), we note that in an $\frac{\epsilon}{2}$ neighborhood of the boundary, only the first term in definition of $\chi_\omega$ is non-zero. Since $\eta=1$ there, we have $\chi_\omega= \Psi$ and it follows that $\p_\nu^k\chi_\omega|_{\p\Omega}=0$ since $\p_{x_n}^k \Psi(x',x_n)=0$ for all $k=1,\ldots$.

\section{} 
\label{J_coerciv}
This section is concerned with verifying the main assumptions in Theorem~\ref{smoothness} for $s=3$, given that Hypothesis~\ref{hypo} holds and that $\chi_\omega$ satisfies properties (i)--(iii) in \eqref{cutoff_conditions}.

We start by proving the mapping property $\chi_\omega: \mathcal D((-\Delta)^3) \to \mathcal D((-\Delta)^3)$ holds. To this end, we recall that the Laplace operator $\Delta$ has the following (well-known) expression in boundary normal coordinates near $\p \Omega$:
\bel{laplace_normal} \Delta = \p^2_{x_n} - a(x',x_n)\p_{x_n} + b(x',x_n)\sum_{i,j=1}^{n-1} \p_{x_i}(c^{ij}(x',x_n)\p_{x_j}),\ee
for some smooth functions $a(x',x_n)$, $b(x',x_n)$ and $c^{ij}(x',x_n)$ near $\p\Omega$. 

Let $(y_0,y_1) \in \mathcal D((-\Delta)^3)$. We need to show that $(\chi_\omega \,y_0,\chi_\omega\,y_1) \in \mathcal D((-\Delta)^3)$. Since $\chi_\omega \in \mathcal C^{\infty}(\Omega)$, it follows that $(\chi_\omega \,y_0,\chi_\omega\,y_1) \in H^{4}(\Omega)\times H^{3}(\Omega)$ and the first property in Definition~\ref{compat} is satisfied. Let us now show that $\chi_\omega \,y_0$ satisfies (ii). First observe that since $y_0=0$ on $\p \Omega$, we have $\chi_\omega y_0=0$ on $\p \Omega$. Next, using equation \eqref{laplace_normal}, together with the fact that $y_0(x',0)=0$, we obtain
\[
\Delta (\chi_\omega y_0) (x',0)=(\p^2_{x_n}(\chi_\omega y_0))(x',0)- a(x',0)(\p_{x_n}(\chi_\omega y_0))(x_n,0).
\]
Since $\p_{x_n}^{k}\chi_\omega(x',0)=0$ for $k=1,2$, this reduces to 
$$\Delta (\chi_\omega y_0) (x',0)=\chi_\omega(x',0)(  (\p^2_{x_n} y_0)(x',0) - a(x',0)(\p_{x_n}y_0)(x',0)).$$
Using again the fact that $y_0(x',0)=0$, this can be recast as 
$$\Delta (\chi_\omega y_0) (x',0)=\chi_\omega(x',0)(\Delta y_0) (x',0)=0,$$
where the last step uses the fact that $\Delta y_0|_{\p \Omega}=0$. This shows that property (ii) in Definition~\ref{compat} also holds for $\chi_\omega y_0$. Analogously, we can show that property (iii) in Definition~\ref{compat} holds for $\chi_\omega y_1$. Thus $\chi_\omega$ maps $\mathcal D((-\Delta)^3)$ and subsequently the first assumption in Theorem~\ref{smoothness} is verified.

Let us now show that $\mathcal J$ is coercive and strictly convex (following \cite[Theorem 2.4]{MZ}). Indeed, since $(0,T)\times \omega_\delta$ satisfies the geometric control condition, Theorem~\ref{observability} applies to obtain the bound
\bel{app_obs} \|U\|_{\mathcal C([0,T];L^2(\Omega))}+ \|\p_tU\|_{\mathcal C([0,T];H^{-1}(\Omega))}  \leq C_0 \|U\|_{(0,T)\times \omega_\delta} \leq C_0\|\sqrt{\chi_\omega}U\|_{L^2(\M)}\ee
for all $U$ that solve the wave equation \eqref{dual}, where $C_0>0$ is a constant depending on $T,\Omega,\omega,\delta$. In the last inequality above, we have used the facts that $\chi_\omega$ is non-negative and that $\chi_\omega=1$ on $\omega_\delta$. 

Applying the Cauchy-Schwarz inequality we write
$$ \mathcal J(U_0,U_1) \geq \frac{1}{2}\|\sqrt{\chi_\omega}U\|_{L^2(\M)}^2 -\|(g_0,g_1)\|_{H^1_0(\Omega)\times L^2(\Omega)}\|(U_0,U_1)\|_{L^2(\Omega)\times H^{-1}(\Omega)}.$$
Now applying the estimate \eqref{app_obs}, it follows that $$ \mathcal J(U_0,U_1) \to \infty\quad \text{as}\quad  \|(U_0,U_1)\|_{L^2(\Omega)\times H^{-1}(\Omega)} \to \infty$$ and therefore, by definition, $\mathcal J$ is coercive. 

To show strict convexity, note that given any pair $(U_0,U_1)$ and $(V_0,V_1)$ in $L^2(\Omega)\times H^{-1}(\Omega)$ and any $\lambda \in [0,1]$ we have
\[
\begin{aligned}
 \mathcal J(\lambda (U_0,U_1)+(1-\lambda)(V_0,V_1)) =\,& \lambda \,\mathcal J(U_0,U_1) + (1-\lambda)\,\mathcal J(V_0,V_1)\\
&- \frac{1}{2}\lambda(1-\lambda)\int_0^T\int_\Omega \chi_\omega |U-V|^2 \,dt\,dx,
\end{aligned}
\]
where $U, V$ solve equation \eqref{dual} with final data $(U_0,U_1)$ and $(V_0,V_1)$ respectively. Now, using \eqref{app_obs} again, we deduce that for $(U_0,U_1) \neq (V_0,V_1)$, there holds:
$$  \mathcal J(\lambda (U_0,U_1)+(1-\lambda)(V_0,V_1)) < \lambda\, \mathcal J(U_0,U_1) + (1-\lambda)\,\mathcal J(V_0,V_1)$$
and therefore the functional $\mathcal J$ is strictly convex on $L^2(\Omega)\times H^{-1}(\Omega)$.

\end{document}